\documentclass[12pt]{article}

\usepackage{graphicx}
\usepackage{latexsym,amssymb}
\usepackage{amsthm}
\usepackage{indentfirst}
\usepackage{amsmath}
\usepackage{color}
\usepackage{xcolor}
\usepackage{fourier}
\usepackage[colorlinks=true]{hyperref}
\usepackage[all]{xy}

\newtheorem*{maintheorem}{Main Theorem}

\newtheorem{Theorem}{Theorem}[section]
\newtheorem*{Theorem A}{Theorem A}
\newtheorem*{Theorem A'}{Theorem A'}
\newtheorem*{Theorem C'}{Theorem C'}

\newtheorem*{Conj*}{Conjecture}

\newtheorem{Definition}[Theorem]{Definition}
\newtheorem{Proposition}[Theorem]{Proposition}
\newtheorem{Lemma}[Theorem]{Lemma}

\newtheorem*{Remark}{Remark}

\newtheorem{Remark-numbered}[Theorem]{Remark}
\newtheorem{Remarks-numbered}[Theorem]{Remarks}
\newtheorem{Corollary}[Theorem]{Corollary}

\newtheorem*{Claim}{Claim}
\newtheorem{Claim-numbered}{Claim}

 \def\NN{{\mathbb N}} 

 \def\RR{{\mathbb R}}

   \def\cN{{\cal N}}

\def\cE{{\cal E}}    
\def\cF{{\cal F}}  \def\cL{{\cal L}}

\newcommand{\id}{\operatorname{Id}}

\newcommand{\sing}{{\operatorname{Sing}}}

\def\dim{\operatorname{dim}}

\def\Sing{\operatorname{Sing}}

\begin{document}

\title{{On the nonlinear Poincar\'e flow}}

\author{Sylvain Crovisier and Dawei Yang\footnote{S.Crovisier was partially supported by \emph{ISDEEC} ANR-16-CE40-0013, and by the ERC project 692925 \emph{NUHGD}. D. Yang  was partially supported by National Key R\&D Program of China (2022YFA1005801), by NSFC 12171348 and NSFC 12325106.
}}


\maketitle


\begin{abstract} 
We develop a tool in order to analyse the dynamics of differentiable flows with singularities.
It provides an abstract model for the local dynamics that can be used in order to control the size of invariant manifolds.
This work is the first part of the results announced in~\cite{CY2}.
\smallskip

\noindent
\emph{2000 Mathematics Subject Classification:} 37C10, 37D30.\\
\emph{Keywords:} Vector field, Poincar\'e flow, dominated splitting, singularity.
\end{abstract}

\section{Introduction}

It is generally believed that the dynamics of vector fields is similar to the dynamics of diffeomorphisms.
Indeed if $X$ is a vector field over a closed Riemannian manifold $M$ which generates the flow $(\varphi_t)_{t\in \RR}$, and if there exists a global cross section $\Sigma$, then the first return map to $\Sigma$
is a diffeomorphism whose dynamics reflects those of the flow. In general there does not exist a global cross section, but Poincar\'e used local cross sections in order to study the dynamics around periodic orbits.
 
When one considers a general compact invariant set $\Lambda$ instead of a periodic orbit, Liao \cite{Lia63} extended the above ideas of Poincar\'e to the regular part of $\Lambda$ and he defined the notion of \emph{linear Poincar\'e flow} on the normal bundle. As summarized by Bonatti-da Luz \cite{BdL}: ``for flows, hyperbolic structures live on the normal bundle for the linear Poincar\'e flow, but not on the tangent bundle''. However, the  linear Poincar\'e flow cannot be defined at singularities.
 
Singularities are special orbits of vector fields: these are those points $\sigma\in M$ such that $X(\sigma)=0$. The other points are said to be \emph{regular} points. Denote by ${\rm Sing}(X)$ the set of singularities of $X$. 
Singularities alone support very simple dynamics: they are in fact the fixed points of the flow. However, when one considers a general compact invariant set with singularities, the regular orbits may accumulate singularities.
This complicates the study of the dynamics on such invariant sets. The famous Lorenz attractor \cite{Lo} is such an example.

The derivative with respect to the space variable $D\varphi_t: TM\to TM$ is said to be the \emph{tangent flow} of $X$. 
For any regular point $x$, let $\cN_x=\{v\in T_x M:~\left<v,X(x)\right>=0\}$
be its normal tangent space.
This defines a normal bundle $\cN_{M\setminus{\rm Sing}(X)}$ on the regular set, which is a priori non-compact. A linear flow $\Psi_t:~\cN_{M\setminus{\rm Sing}(X)}\to \cN_{M\setminus{\rm Sing}(X)}$ is then obtained in the following way: for any vector $v\in\cN_x$ with $x\in M\setminus{\rm Sing}(X)$
$$\Psi_t\cdot v=D\varphi_t\cdot v-\frac{\left<X(\varphi_t(x)),D\varphi_t\cdot v\right>}{\left|X(\varphi_t(x))\right|^2}X(\varphi_t(x)).$$

Due to the existence of singularities, some compactness will be lost. In this paper, we use techniques for studying flows with singularities that may be useful for
other problems.

In order to analyze the tangent dynamics (for instance in order to prove the existence of a dominated splitting,
and then build invariant manifolds), one needs to analyze the local dynamics near $\Lambda$.
For a diffeomorphism $f$, one usually lifts the local dynamics to the tangent bundle:
for each $x\in M$, one defines a diffeomorphism $\cL f_x\colon T_xM\to T_{f(x)}M$,
which preserves the $0$-section (i.e. $\cL f_x(0_x)=0_{f(x)}$) and is locally conjugated to $f$
through the exponential map:
$\exp_{f(x)}\circ \cL f_x=f\circ \exp_x$ near $0_x$ in $T_xM$.
 It defines in this way a local fibered system on the bundle $TM\to M$.
For flows one introduces a similar notion.

\begin{Definition}\label{d.local-flow}
Let $(\varphi_t)_{t\in \RR}$ be a continuous flow over a compact metric space $K$, and $\cN\to K$ be a continuous Riemannian vector bundle.
A \emph{local $C^k$ fibered flow} $\psi$ on $\cN$ is a continuous family of $C^k$ diffeomorphisms $\psi_{t}\colon \cN_x\to \cN_{\varphi_t(x)}$, for $(x,t)\in K\times \RR$, such that:
\begin{itemize}
\item[--] Each map $\psi_t$ preserves the $0$-section of $\cN$.
\item[--] The flow property holds near the $0$-section:
there is $\beta_0>0$ such that for each $x\in K$, $t_1,t_2\in \RR$, and $u\in \cN_x$ satisfying
$\big[\forall s\in[0,1], \max(\|\psi_{s.t_1}(u)\|, \|\psi_{s.t_2}(\psi_{t_1}(u))\|)\leq \beta_0\big]$, 
then
\begin{equation}\label{e.identification}
\psi_{t_1+t_2}(u)=\psi_{t_2}\circ \psi_{t_1}(u).
\end{equation}
\end{itemize}
\end{Definition}

\begin{Definition}\label{Def:identification}
A \emph{$C^k$ identification} $\mathfrak{h}$ on an open set $U\subset K$
is a continuous family of $C^k$ diffeomorphisms $\mathfrak{h}_{y,x}\colon \cN_y\to \cN_x$ indexed by
pairs of close points $x,y \in U$, satisfying for some $\beta_0,r_0>0$ the following: for any $x,y,z\in U$ and $u\in \cN_y$
with $d(z,x),d(z,y)<r_0$ and  $\|u\|<\beta_0$,
$$\mathfrak{h}_{z,x}\circ \mathfrak{h}_{y,z}(u)=\mathfrak{h}_{y,x}(u).$$
\end{Definition}
We say that the identification $\mathfrak{h}$ is \emph{compatible with the local fibered flow}  $(\psi_t)_{t\in \RR}$, if it satisfies some commutation property near the $0$-section, see Section~\ref{s.identifications} for the formal definition.

For a vector field $X$, a natural way to lift the dynamics is to define a \emph{Poincar\'e map} by projecting the normal spaces $\cN_x$ and
$\cN_{\varphi_t(x)}$ above two points of a regular orbit by using the exponential map and the holonomy along the orbits of the flow.
One obtains a local diffeomorphism $\psi_t$ from a neighborhood of $0_x$ in $\cN_x$ to
$\cN_{\varphi_t(x)}$. The advantage of this construction is that the ambient dimension has been dropped by $1$. The linearization of the Poincar\'e map $\psi_t$ at $0_x$ coincides with the linear Poincar\'e flow $\Psi_t$.

A new difficulty appears: the domain of the Poincar\'e maps degenerate near the singularities.
For that reason one introduces the \emph{rescaled Poincar\'e flow}:
$$\psi^*_t(u)=\|X(\varphi_t(x))\|^{-1}\cdot \psi_t(\|X(x)\|\cdot u).$$
We denote it by $\psi^*_{t,x}$ if we want to indicate the base point $x$.

Our main result is that this family of maps can be compactified
as a local fibered flow, assuming that the singularities are not degenerate.

\begin{maintheorem}
Let $X$ be a $C^k$ vector field, $k\geq 1$, on a closed manifold $M$ of dimension $d$ and $\Lambda$ be a compact set, invariant by the flow $(\varphi_t)_{t\in\RR}$
associated to $X$, such that $DX(\sigma)$ is invertible for each  $\sigma\in {\rm Sing}(X)\cap \Lambda$.
Let $(\psi^*_t)_{t\in \RR}$ be the rescaled nonlinear Poincar\'e flow on $\cN|_{\Lambda\setminus {\rm Sing}(X)}$.

Then, there exist a topological flow $\widehat \varphi$ on a compact metric space $\widehat \Lambda$,
and a local $C^k$ fibered flow $\widehat {\psi^*}$ on a Riemannian vector bundle $\widehat {\cN}\to \widehat \Lambda$
whose fibers have dimension $d-1$ such that:
\begin{itemize}
\item[--] there is an injective map $i\colon \Lambda\setminus {\rm Sing}(X)\to \widehat \Lambda$
whose image is open in $\widehat\Lambda$, and which conjugates the restrictions
of $\varphi$ to $\Lambda\setminus {\rm Sing}(X)$ and of $\widehat \varphi$
to $i(\Lambda\setminus {\rm Sing}(X))$;

\item[--] there is an isometric bundle isomorphism $I\colon \cN|_{\Lambda\setminus{\rm Sing}(X)}\to \widehat \cN|_{i(\Lambda\setminus{\rm Sing}(X))}$ fibered over $i$;

\item[--] $I$ conjugates, near the $0$-section, the rescaled nonlinear Poincar\'e flow $\psi^*$
over $\Lambda\setminus {\rm Sing}(X)$ and the restriction of the local fibered flow $\widehat {\psi^*}$
to $i(\Lambda\setminus {\rm Sing}(X))$. Locally:
$\widehat {\psi^*}=I\circ \psi^*\circ I^{-1}$.
\end{itemize}
Moreover for any open set $\widehat U\subset \widehat \Lambda$
whose closure is disjoint from 
$i(\Lambda\setminus{\rm Sing}(X))$, there exist $C>0$ and a $C^k$ identification $\widehat{\mathfrak{h}}$ on $\widehat U$
which is compatible with the flow $(\widehat {\psi^*_{C\cdot t}})_{t\in\RR}$.
\end{maintheorem}

Our main goal initially was to prove a conjecture by Palis for $3$-dimensional flows (see the announcement in~\cite{CY1}).
The present work is the first part of the proof of this conjecture (and is extracted from our previous unpublished preprint~\cite{CY2}); the other parts of the proof of the conjecture will appear in subsequent papers. The Main Theorem here is essential in order to prove the existence of stable manifolds that are weakly contracted, with uniform $C^2$ topology, allowing to get distortion control.

As we mentioned the local fibered flow $(\psi_t)_{t\in \RR}$ degenerates near the singularities.
The rescaling solves the degeneracy. From the proof of the Main Theorem, we get in particular the following property (see Theorem~\ref{t.non-degeneracy}):
\emph{In the setting of the Main Theorem, and given $t,\varepsilon>0$, there is $\delta,\beta>0$ such that for any regular points $x,y\in M$, if $\RR.X(x)$ and $\RR.X(y)$ are $\delta$-close, then $\psi_{t,x}^*|_{B(0_x,\beta)}$ and $\psi_{t,y}^*|_{B(0_y,\beta)}$ are $\varepsilon$-close in the $C^k$ topology
on the .}

Our result extends previous constructions.
Liao first introduced the rescaling~\cite{Lia89} of the linear Poincar\'e flow. For every regular point $x\in M$ and every $v\in\cN_x$, one defines
$$\Psi_t^*\cdot v=\frac{\|X(x)\|}{\|X(\varphi_t(x))\|}\Psi_t\cdot v=\frac{\Psi_t\cdot v}{\|D\varphi_t(x)|_{\RR\cdot X(x)}\|}.$$
It is clear that the linearization of the rescaled Poincar\'e map $\psi^*_t$ is $\Psi^*_t$.
Note that the hyperbolicity properties of $\Psi^*$ allows to detect the existence of dominated splittings
for the tangent flow $D\varphi$: 
a subbundle of $\cN$ is uniformly contracted by $\Psi^*$ if and only if
the dynamics by $\Psi$ is dominated by $D\varphi|_{\RR\cdot X}$
(we call this property: mixed dominated splitting).
Consequently, \emph{if $\cN$ has a dominated splitting $\cN=\cE\oplus \cF$ for $\Psi$
(or equivalently for $\Psi^*$), and if  $\cE$ is uniformly contracted by $\Psi^*$,
then $TM$ has a dominated splitting $TM=E\oplus F$ such that $\dim(E)=\dim(\cE)$
and $X$ is tangent to $F$}, see~\cite[Lemma 5.5]{lgw-extended} and~\cite[Lemma 2.13]{GY}.

The introduction of the compactification of the rescaled non-linear Poincar\'e flow gives a general
framework in order to study locally flows with singularities.
Gan and Yang~\cite{GY} also considered previously the rescaled non-linear Poincar\'e flow and proved some uniform properties about the size of invariant manifolds in the $C^1$-topology. Our construction allows to recover their results in a different way
which avoids long computations and to extend it to the $C^k$-topologies. As an example of application, these techniques allow to prove that,
\emph{at points $x$ satisfying some Pliss-hyperbolicity condition,  there exist stable manifolds with size of order
$d(x, \Sing(X))$}, see~\cite[Lemma 2.18]{GY}.

As another application, our compactification allows to obtain the proof of a property for vector fields
which generalizes a classical result for diffeomorphisms, see Section~\ref{Sec:negative-periodic} below. Versions of this result have been proved
in~\cite{LY} for $C^{2}$ vector fields and in~\cite[Lemma 2.20]{Y} for $C^1$ vector fields on surfaces:
\emph{let $X$ be a $C^1$ vector field on a closed manifold $M$ such that $DX(\sigma)$ is invertible at each singularity $\sigma$ and let $\mu$ be an ergodic non-atomic probability measure with $d-1$ negative Lyapunov exponents. Then $\mu$ is supported on an attracting periodic orbit.}

\paragraph{\bf Notations.}
During the text we adopt the following conventions for the notations of the different flows:
\begin{enumerate}
\item \emph{Linear flows} (on vector bundles) are denoted by capital letters $\Phi,\Psi,...$,
whereas nonlinear flows are denoted by lowercase letters: $\varphi,\psi,...$
\item \emph{Rescaled flows} (linear or nonlinear) are marked with a star: $\Phi^*,\Psi^*,\psi^*,...$
\item \emph{Grassmann dynamics:} flows $\varphi$ or bundles $\cE$ on $M$ will be lifted to the unit tangent bundle $T^1M$ and projective tangent bundle $PTM$, and denoted with a tilde: $\widetilde \varphi$, $\widetilde \cE$.
\item \emph{Compactifications:} flows $\varphi$ or bundles $\cE$ on $M$ will be lifted to some blowups $\widehat M$ of $M$,
that will be called \emph{extended flows} or \emph{extended bundles} and denoted with a hat: $\widehat\varphi$, $\widehat \cE$.
\end{enumerate}
Some of the constructions done in Sections~\ref{s.grassmann} and~\ref{s.blow-up} below can be summarized as follows:
\begin{displaymath}
\xymatrix{
&(\widehat \cN, \widehat \Psi)\ar[d]
&(\cN,\Psi)\ar@{^{(}->}[r]^{\widetilde I}\ar@{_{(}->}[l]_{I} \ar[d]
&(\widetilde \cN, \widetilde \Psi)\ar[d]&\\
(\widehat {TM},\widehat \Phi) \ar[ur]^{\widehat Q} \ar[d]^{P} \ar[r]^{\widehat \pi}
&(\widehat M,\widehat \varphi)  \ar[d]^{p}
&(M\setminus {\rm Sing}(X),\varphi) \ar@{^{(}->}[r]^{\widetilde{\text{\it \i}}}\ar@{_{(}->}[l]_{\hspace{-0.5cm}i}
&(\widetilde M=PTM,\widetilde \varphi)  \ar[d]^{\widetilde p}
&(\widetilde {TM},\widetilde \Phi) \ar[ul]_{\widetilde Q} \ar[d]^{\widetilde P} \ar[l]_{\widetilde \pi}\\
(TM,\Phi=D\varphi) \ar[r]^{\quad \pi}
&(M,\varphi)
&&(M,\varphi)
&(TM,\Phi=D\varphi) \ar[l]_{\hspace{-0.3cm}\pi}
}
\end{displaymath}

\section{Dynamics on Grassmann bundles}\label{s.grassmann}
We associate to the flow $(\varphi_t)_{t\in\RR}$ several $C^{k-1}$ linear and projective flows.

\subsection{The \emph{tangent flow} $(\Phi_t)_{t\in\RR}=(D\varphi_t)_{t\in \RR}$ on the tangent bundle $TM$}
The tangent flow is the flow on the tangent bundle
$TM$ which fibers over $(\varphi_t)_{t\in\RR}$ through the projection
$\pi\colon TM\to M$ and is obtained by differentiation.

\subsection{The \emph{Grassmann flow} $(\widetilde \varphi_t)_{t\in\RR}$ on $PTM$}
We consider the projective tangent manifold $\widetilde M=PTM$, which is the Grassmannian space of
$1$-dimensional linear spaces $L_x\subset T_xM$ at points $x\in M$
and let $\widetilde p\colon PTM\to M$ denote the canonical projection.

The projective tangent flow on
$PTM$ is obtained from $(D\varphi_t)_{t\in\RR}$ by taking the quotient:
$$\widetilde \varphi_t(\RR\cdot v)=\RR \cdot (D\varphi_t\cdot v) \text{ for non vanishing} v\in TM.$$
Sometimes we prefer to work with the unit tangent manifold $T^1M$ which is a double cover of $PTM$. The tangent flow induces by normalization a flow on this space, that we also denote by  $(\widetilde \varphi_t)_{t\in\RR}$ for simplicity.

\subsection{The linear flow $(\widetilde \Phi_t)_{t\in\RR}$ on $\widetilde{TM}$}
The projection $\widetilde p\colon PTM\to M$ allows
to pull back the tangent bundle $\pi\colon TM\to M$ and induces
a linear bundle $\widetilde \pi\colon \widetilde{TM}\to PTM$.
It can be obtained as the restriction of the first projection
$PTM\times TM\to PTM$ to the set of pairs
$(x,v)$ such that $\widetilde p(x)=\pi(v)$.
It is naturally endowed with the pull back metric of $TM$ and it is locally trivial in a neighborhood of preimages
$\widetilde p^{-1}(z)$ for $z\in M$.
The second projection $PTM\times TM\to PTM$
induces a projection $\widetilde P \colon \widetilde{TM}\to TM$ which is an isometry on each fiber.

By pulling back the tangent flow $D\varphi$, one defines a linear flow
$(\widetilde \Phi_t)_{t\in\RR}$ on $\widetilde{TM}$ over $(\widetilde \varphi_t)_{t\in\RR}$.

\subsection{The \emph{normal flow} $(\widetilde \Psi_t)_{t\in\RR}$ on $\widetilde{\cN}$}
We define the normal bundle $\widetilde \cN\to PTM$:
any element of $PTM$ is a one-dimensional line $L\subset T_xM$
and its fiber in $\widetilde \cN$ is the orthogonal subspace of $\widetilde{TM}_L$
to $L$:
$$\widetilde \cN_L:=\{u\in \widetilde{TM}_L: \widetilde P(u)\perp L\}.$$
Let us consider the orthogonal quotient
$\widetilde Q\colon \widetilde {TM}\to \widetilde N$ . For any $u\in \widetilde \cN_L$,
\begin{equation}\label{e.projection}
\widetilde Q(u):=u-\frac{<u, v>}{\|v\|^2}\cdot v,
\end{equation}
where $v$ is any non-zero vector in $\widetilde{TM}$
such that $\widetilde P(v)\in L$.

The flow $(\widetilde \Phi_t)_{t\in\RR}$ induces a linear flow $(\widetilde \Psi_t)_{t\in\RR}$ on the quotient:
$$
\widetilde \Psi_t\cdot u:=(\widetilde p\circ \widetilde \Phi_t)\cdot u.$$

\subsection{The \emph{linear Poincar\'e flow} $(\Psi_t)_{t\in\RR}$ on the normal bundle $\cN$}
\label{ss.linear-Poincare-flow}
The normal bundle $\cN\to M\setminus {\rm Sing}(X)$ over the space of non-singular points $x$
is the union of the vector subspaces $\cN_x=X(x)^\perp$,
which can be seen as the image of an orthogonal projection
$q\colon TM\to \cN$ (defined in a similar way as in~\eqref{e.projection}).
Contrary to the previous bundles, $\cN$ depends on the vector field $X$
and is as smooth as $X$.

Using the vector field one lifts $M\setminus {\rm Sing}(X)$ to $PTM$ by the map
(a local section of $\widetilde p$):
$$\widetilde{\text{\it \i}}\colon x\mapsto (x,\RR\cdot X(x)).$$
In this way the spaces $\cN_x$ can be identified with $\widetilde \cN_{\widetilde{\text{\it \i}}(x)}$ by $\widetilde P$.
Let us denote $\widetilde I\colon \cN_x \to \widetilde \cN_{\widetilde{\text{\it \i}}(x)}$ the identification induced by $\widetilde P^{-1}$.

\section{Blowup}\label{s.blow-up}

We will consider compactifications of $M\setminus {\rm Sing}(X)$ and of the tangent bundle
$TM|_{M\setminus \text{Sing}(X)}$ which allow to extend the line field $\RR X$. This is given by the classical blowup.
In this section we assume that $DX(\sigma)$ is invertible at each singularity $\sigma$. In particular ${\rm Sing}(X)$
is a finite set.

\subsection{The blownup manifold $\widehat M$}
We can blow up $M$
at each singularity of $X$ and get a new closed manifold $\widehat M$
and a projection $p\colon \widehat M\to M$ which is one-to-one above $M\setminus {\rm Sing}(X)$.
Each singularity $\sigma\in{\rm Sing}(X)$ has been replaced by the projectivization
$PT_\sigma M$. We denote $i\colon M\setminus {\rm Sing}(X)\to \widehat M$
the inverse of $p$ on the regular set.

More precisely, at each singularity $\sigma$,
one add $T^1_\sigma M$ to $M\setminus \{\sigma\}$ in order
to build a manifold with boundary. Locally,
it is defined by the chart
$[0, \varepsilon)\times T^1_\sigma M\to (M\setminus \{\sigma\})\cup T^1_\sigma M$
given by:
$$ (s,u)\mapsto 
\begin{cases}
&\exp(s\cdot u) \text{ if }s\neq 0,\\
&u
\text{ if } s=0.
\end{cases}
$$
One then gets $\widehat M$ by identifying points $(0,u)$ and $(0,-u)$
on the boundary.

It is sometimes convenient to
lift the dynamics on $M\setminus \{\sigma\}$ near $\sigma$
and work in the local coordinates
$(-\varepsilon, \varepsilon)\times T^1_\sigma M$.
These coordinates define a double covering of an open subset of the blowup $\widehat M$
and induce a chart from the quotient $(-\varepsilon, \varepsilon)\times T^1_\sigma M /_{(s,u)\sim(-s,-u)}$
to a neighborhood of $p^{-1}(\sigma)$ in $\widehat M$.

\subsection{The \emph{extended flow $(\widehat \varphi_t)_{t\in \RR}$} on $\widehat M$}
The following result is proved in~\cite[Section III]{takens}.
\begin{Proposition}
The flow $(\varphi_t)_{t\in \RR}$ induces a $C^{k-1}$ flow $(\widehat \varphi_t)_{t\in\RR}$
on $\widehat M$ which is associated to a $C^{k-1}$ vector field $\widehat X$. 
For $\sigma\in {\rm Sing}(X)$, this flow preserves  $PT_\sigma M$, and acts as:
$$\widehat \varphi_t(u)=\tfrac{D\varphi_t(0)\cdot u}{\|D\varphi_t(0)\cdot u\|}.$$
The vector field $\widehat X$ coincides at $u\in PT_\sigma M$ with $DX(\sigma).u$
in $T_{u}(PT_\sigma M)\cong T_\sigma M/(\RR\cdot u)$.
\end{Proposition}
In particular the tangent bundle $T\widehat M$ extends $TM|_{M\setminus \text{Sing}(X)}$,
the linear flow  $D\widehat \varphi$ on $T\widehat M$ extends the restriction of $D\varphi$
to $TM|_{M\setminus \text{Sing}(X)}$
and the vector field $\widehat X$  extends $X$.
Note that each eigendirection $u$ of $DX(\sigma)$ at a singularity $\sigma$
induces a singularity of $\widehat X$.

\begin{Remark}
In~\cite{takens}, the vector field and the flow are extended locally on the space $(-\varepsilon, \varepsilon)\times T^1_\sigma M$,
but the proof shows that these extensions are invariant under the map $(s,u)\mapsto (-s,-u)$, hence are also defined on $\widehat M$.
\end{Remark}

\subsection{The \emph{extended bundle} $\widehat {TM}$
and \emph{extended tangent flow} $(\widehat {\Phi}_t)_{t\in \RR}$}
One associates a bundle $\widehat \pi\colon \widehat {TM}\to \widehat M$
which is the pull-back of the bundle $\pi\colon TM\to M$  over $M$
by the map $p\colon \widehat M\to M$.
It can be obtained as the restriction of the first projection
$\widehat M\times TM\to \widehat M$ to the set of pairs
$(x,v)$ such that $p(x)=\pi(v)$.
It is naturally endowed with the pull back metric of $TM$ and it is locally trivial in a neighborhood of preimages
$p^{-1}(z)$, $z\in M$.
Note that the linear bundle $\widehat{TM}$ is different from the tangent bundle
$T\widehat M$.

The tangent flow $(D\varphi_t)_{t\in \RR}$ can be pull back to $\widehat {TM}$
as a $C^{k-1}$ linear flow $(\widehat {\Phi}_t)_{t\in \RR}$
that we call \emph{extended tangent flow}.

\subsection{The \emph{extended line field} $\widehat {\RR X}$}
The vector field $X$  induces a line field ${\RR X}$ on $M\setminus \text{Sing}(X)$
which admits an extension to $\widehat {TM}$. It is defined locally as follows.

\begin{Proposition}\label{p.extended-field}
At each singularity $\sigma$, let $U$ be a small neighborhood in $M$
and $\widehat U=(U\setminus \{\sigma\})\cup PT_\sigma M$ be a neighborhood of $PT_\sigma M$
in $\widehat M$.
Then, {the map $x\mapsto \frac{\exp_\sigma^{-1}(x)}{\|X(x)\|}$ on $U\setminus\{\sigma\}$ extends to $\widehat U$ as a $C^{k-1}$-map which coincides at $u\in PT_\sigma M$ with $\frac{u}{\|DX(\sigma)\cdot u\|}$,} and 
the map $x\mapsto \frac{\|X(x)\|}{\|\exp_\sigma^{-1}(x)\|}$ on $U\setminus \{\sigma\}$
extends to $\widehat U$ as a $C^{k-1}$-map which coincides at $u\in PT_\sigma M$
with $\|DX(\sigma)\cdot u\|$.

In the local coordinates $(-\varepsilon, \varepsilon)\times T^1_\sigma M$ associated to $\sigma\in \text{Sing}(X)$,
the lift of the vector field $X_1:=X/{\|X\|}$ on $M\setminus \text{Sing}(X)$ extends as a (non-vanishing) $C^{k-1}$ section
$\widehat X_1\colon (-\varepsilon, \varepsilon)\times T^1_\sigma M\to \widehat{TM}$.
For each $x=(0,u)\in p^{-1}(\sigma)$,
one has $$\widehat X_1(x)=\frac{DX(\sigma)\cdot u}{\|DX(\sigma)\cdot u\|}.$$
\end{Proposition}

The proposition gives an extension of $X_1$ to the local chart $(-\varepsilon, \varepsilon)\times T^1_\sigma M$
which is a double covering of a neighborhood of $p^{-1}(\sigma)$ in $\widehat M$.
A priori, the extension of $X_1$ is not preserved by the symmetry $(-s,-u)\sim(s,u)$
and does not define a map $\widehat M\to \widehat{TM}$.
However, the line field $\RR \widehat X_1$ is invariant by the local symmetry $(s,u)\mapsto (-s,-u)$,
hence induces a $C^{k-1}$-line field $\widehat {\RR X}$ on $\widehat {TM}$ invariant by $(\widehat {D\varphi_t})_{t\in \RR}$.

\begin{proof}
In a local chart near a singularity, we have
$$X(x)=\int_{0}^1(DX(r\cdot x)\cdot x)\;dr.$$
Working in the local coordinates $(s,u)\in(-\varepsilon,\varepsilon)\times T^1_\sigma M$, we get
$$X(x)=\bigg(\int_{0}^1DX(r\cdot s\cdot u)\;dr\bigg)\;. \;(s\cdot u).$$
This allows us to define a $C^{k-1}$ section
in a neighborhood of $p^{-1}(\sigma)$ defined by
$$\overline X\colon(s,u)\mapsto  \bigg(\int_{0}^1DX(r\cdot s\cdot u)\;dr\bigg)\cdot  u.$$
This section is $C^{k-1}$, is parallel to $X$ (when $s\neq 0$)
and does not vanish.
Consequently $\frac{\overline X}{\|\overline X\|}$ is $C^{k-1}$ and extends the vector field
$X_1:=X/{\|X\|}$ as required.

Since $\overline X$ extends as $DX(\sigma)\cdot u$ at $u\in PT_\sigma M$,
then $X_1$ extends as $DX(\sigma)\cdot u/ \|DX(\sigma)\cdot u\|$.

Note also that for $s\neq 0$,
$\|\overline X(s,u)\|$ coincides with $\|X(x)\|/\|\exp_\sigma^{-1}(x)\|$ where $x=\exp_\sigma(s\cdot u)$ is a point
of $M\setminus \{\sigma\}$ close to $\sigma$.
Since $\overline X$ is $C^{k-1}$ and does not vanish,
$(s,u)\mapsto \|\overline X(s\cdot u)\|$ extends as a $C^{k-1}$-function in the local coordinates
$(-\varepsilon,\varepsilon)\times T^1_\sigma M$. It is invariant by the symmetry
$(s,u)\sim (-s,-u)$, hence the maps {$x\to \frac{\exp_\sigma^{-1}(x)}{\|X(x)\|}$ and} $x\mapsto \|X(x)\|/\|\exp_\sigma^{-1}(x)\|$
for $x\in M\setminus \{\sigma\}$ close to $\sigma$ extends
as $C^{k-1}$ maps on a neighborhood of $PT_\sigma M$ in $\widehat M$.
\end{proof}

\subsection{\emph{Extended normal bundle} $\widehat {\cN}$ and \emph{extended linear Poincar\'e flow} $(\widehat \Psi_t)_{t\in \RR}$}
The orthogonal spaces to the lines of $\widehat {\RR X}$ define
a $C^{k-1}$ linear subbundle $\widehat {\cN}$ of $\widehat{TM}$. One denotes
by $\widehat Q\colon \widehat{TM}\to \widehat{\cN}$ the orthogonal projection. Since
$\widehat {\RR X}$ is preserved by the extended tangent flow, the projection of $(\widehat {D\varphi_t})_{t\in \RR}$
defines the $C^{k-1}$ \emph{extended linear Poincar\'e flow} $(\widehat \Psi_t)_{t\in \RR}$ on $\widehat {\cN}$.

\subsection{Alternative construction}
As we saw in Section~\ref{ss.linear-Poincare-flow},
one can embed $M\setminus {\rm Sing}(X)$ in $PTM$ by the map
$\widetilde{\text{\it \i}}\colon x\mapsto (x,\RR \cdot X(x))$. Taking the closure in $PTM$,
one gets a set which is invariant by the projective tangent flow $\widetilde \varphi$.
This way to compactify $M\setminus {\rm Sing}(X)$
depends on the vector field $X$ and not only on the finite set ${\rm Sing}(X)$.
It is sometimes called \emph{Nash blowup}, see~\cite{No}.

When $DX(\sigma)$ is invertible at each singularity,
Proposition~\ref{p.extended-field} shows that the closure is homeomorphic to $\widehat M$.
Indeed the continuous map
$(-\varepsilon,\varepsilon)\times T^1_\sigma M\to T^1M$
defined by $(s,u)\mapsto (s\cdot u, X(s\cdot u)/\|X(s\cdot u)\|)$ is injective
and descends as an injective quotient map from a neighborhood of $p^{-1}(\sigma)$
to $PTM$.

The restriction of the normal bundle $\widetilde \cN\to PTM$ to the closure of $M\setminus {\rm Sing}(X)$ in $PTM$
gives the extended normal bundle $\widehat {\cN}$.
The extended linear Poincar\'e flow $(\widehat \Psi_t)_{t\in \RR}$ is the restriction of
$(\widetilde \Psi_t)_{t\in\RR}$ to the extended normal bundle $\widehat {\cN}$.
This is the approach followed in~\cite{lgw-extended} in order to compactify of the
linear Poincar\'e flow.

\begin{Remark}
As said previously, the vector field $X$ extends as a vector field $\widehat X$
on $\widehat M$, and still admits singularities $\widehat \sigma$
Assuming that $\widehat \sigma$ is unique on a neighborhood $U$,
one may wonder if the line field $\RR \widehat X$ which is defined
on the tangent bundle to $U\setminus \{\widehat \sigma\}$ extends to $U$
since the limit set of $\RR \widehat X$ at $\widehat \sigma$ may
take all the eigendirection of $D\widehat \varphi_t(\widehat \sigma)$.
This is the reason why, in order to extend $\RR \widehat X$, one has to work
with the pull-back bundle $\widehat{TM}$ and not the tangent bundle
$T\widehat M$.
\end{Remark}

\section{Nonlinear local flows}\label{ss.def-flow}
We now define non-linear flows on neighborhoods of the $0$-section of $TM$ and $\cN$.

\subsection{The \emph{lifted flow} $(\cL\varphi_t)_{t\in\RR}$ on $TM$}
There exists $r_0>0$ such that the ball $B(0,r_0)$ in
each tangent space $T_xM$ projects in $M$ diffeomorphically
by the exponential map.
For each $t\in [0,1]$ and $x\in M$, the map
$$\cL\varphi_t\colon y\mapsto \exp^{-1}_{\varphi_t(x)}\circ \varphi_t\circ \exp_x(y)$$
sends diffeomorphically a neighborhood of $0$ in $T_xM$ to a neighborhood of $0$
in $T_{\varphi_t(x)}M$.
This extends to a local flow $(\cL\varphi_t)_{t\in\RR}$ in a neighborhood of the $0$-section of $TM$, that is called \emph{lifted flow}. It is tangent to $(D\varphi_t)_{t\in\RR}$ at the $0$-section, i.e. $D(\cL\varphi_t)(0_x)=D\varphi_t(x)$.

\subsection{The \emph{fiber-preserving lifted flow} $(\cL_0\varphi_t)_{t\in\RR}$ on $TM$}
Sometimes we will need to preserve the base point and we introduce a
family of fiber-preserving maps $(\cL_0\varphi_t)_{t\in\RR}$, defined by:
$$\cL_0\varphi_t(y)=\exp^{-1}_{x}\circ \varphi_t\circ \exp_x(y).$$
Since the $0$-section is not preserved, this is not a local flow and it will
be considered only for short times $t$. This family is called \emph{fiber-preserving lifted flow}.

\subsection{The \emph{nonlinear Poincar\'e flow} $(\psi_t)_{t\in\RR}$ on $\cN$}\label{ss.NLPF}
Using the exponential map and the transversality between $\cN_x$ and $X(x)$,
one gets each $x\in M\setminus {\rm Sing}(X)$ some number $r_x\in (0,r_0)$
such that for any $t\in [0,1]$, the holonomy map of the flow induces a local diffeomorphism
$\psi_t$ from $B(0,r_x)\subset \cN_x$ to a neighborhood of  $0$ in $\cN_{\varphi_t(x)}$.

This extends to a local flow $(\psi_t)_{t\in\RR}$ in a neighborhood of the $0$-section
in $\cN$, that is called \emph{nonlinear Poincar\'e flow}. It is tangent to
$(\Psi_t)_{t\in\RR}$ at the $0$-section of $\cN$, i.e. $D\psi_t(0_x)=\Psi_t(x)$.
\medskip

The following lemmas give estimates on the radius $r_x$.
They follow from the fact that  the vector field $X$ is almost constant in the {$\beta\cdot \|X(y)\|$-neighborhood of $y$} for $\beta>0$ small enough. It is essentially contained in~\cite[Lemma 2.3]{GY} and in the proof there.

\begin{Lemma}\label{l.project}
For any $\delta,t_0>0$, there exists $\beta>0$ with the following property.

Let $x \in M\setminus \sing(X)$ and $y\in M$ such that $d(x,y)\leq \beta\cdot \|X(x)\|$.
Then,  there is a unique $t\in (-t_0,t_0)$
such that $\varphi_t(y)$ belongs to the image by $\exp_x$ of $B(0_x,\delta{\|X(x)\|})\subset \cN_x$.
\end{Lemma}

\begin{Lemma}\label{l.flow}
For any $\rho>0$, there exists $\delta>0$ with the following property.

Let $x \in M\setminus \sing(X)$ and
$u\in B(0_x,\delta\cdot \|X(x)\|)\subset \cN_x$.
Then for any $t\in (0,1)$, there is a unique $t'>0$ such that
$\varphi_{t'}(\exp_x(u))=\exp_{\varphi_t(x)}\circ \psi_t(u)$.
Moreover $\max\big(\tfrac t {t'},\tfrac{t'} t\big)<1+\rho$.
\end{Lemma}

\section{Rescaled flows}
In order to compensate the effect of the singularities on the transverse behavior of the flow,
one rescales the bundles.

\subsection{\emph{Rescaled linear} and \emph{nonlinear Poincar\'e flows}
$(\Psi^*_t)_{t\in\RR}$, $(\psi^*_t)_{t\in\RR}$ on $\cN$}\label{ss.rescaled}
By Lemma~\ref{l.flow}, there exists
$\beta>0$ such that 
the radii $r_x$ introduced in Section~\ref{ss.NLPF} to define $\psi$ satisfies,
at any $x\in M\setminus {\rm Sing}(X)$,
$$r_x>\beta \|X(x)\|.$$
We can thus rescale the nonlinear Poincar\'e flow. We get for each $x\in M\setminus {\rm Sing}(X)$ and $t\in [0,1]$ a map $\psi^*_t$ which sends
diffeomorphically $B(0,\beta)\subset \cN_x$ inside $\cN_{\varphi_t(x)}$, defined by:
$$\psi^*_t(y)=\|X(\varphi_t(x))\|^{-1}\cdot \psi_t(\|X(x)\|\cdot y).$$
Again, this induces a local flow $(\psi_t^*)_{t\in\RR}$ in a neighborhood of the $0$-section
in $\cN$, that is called \emph{rescaled nonlinear Poincar\'e flow}.
Its tangent map at the $0$-section defines the \emph{rescaled linear Poincar\'e flow} $(\Psi_t^*)$. Thus,
$$\Psi_t^*\cdot v=\frac{\|X(x)\|}{\|X(\varphi_t(x))\|}\Psi_t\cdot v.$$

\subsection{\emph{Rescaled tangent flow} $(\Phi_t^*)_{t\in\RR}$ and \emph{rescaled lifted flow} $(\cL\varphi_t^*)_{t\in\RR}$ on $TM$}
The \emph{rescaled lifted flow} is defined on a neighborhood of the $0$-section
in $TM$ by
$$\cL\varphi^*_t(y)=\|X(\varphi_t(x))\|^{-1}\cdot \cL\varphi_t(\|X(x)\|\cdot y).$$
Its tangent map at the $0$-section defines the \emph{rescaled tangent flow} $(\Phi_t^*)_{t\in\RR}$:
$$\Phi_t^*\cdot v=\frac{\|X(x)\|}{\|X(\varphi_t(x))\|}D\varphi_t\cdot v.$$

\subsection{The \emph{rescaled fiber-preserving lifted flow} $(\cL_0\varphi_t^*)_{t\in\RR}$ on $TM$}
The flow is defined similarly:
$$\cL_0\varphi^*_t(y)=\|X(x)\|^{-1}\cdot \cL_0\varphi_t(\|X(x)\|\cdot y).$$

\section{Compactifications of nonlinear local fibered flows}
The rescaled flows introduced above extend to the bundles $\widehat {TM}$
or $\widehat {\cN}$.
In the following, one will assume that $DX(\sigma)$ is invertible at each singularity.

{Related to the ``local $C^k$ fibered flow''} (Definition~\ref{d.local-flow}),
we will use the following notion.

\begin{Definition}
Consider a continuous Riemannian vector bundle {$\cN$ over a compact metric space $K$.}
A map $H\colon \cN\to \cN$ is \emph{$C^k$ fibered}, if it fibers over a homeomorphism $h$ of $K$
and if each induced map $H_x\colon \cN_x\to \cN_{h(x)}$ is $C^k$
and depends continuously on $x$ for the $C^k$ topology.
\end{Definition}
\subsection{The \emph{extended lifted flow}}
The following proposition compactifies the rescaled lifted flow $(\cL\varphi^*_t)_{t\in \RR}$
(and the rescaled tangent flow $(\Phi_t^*)_{t\in \RR}$) as
local fibered flows on $\widehat {TM}$.
\begin{Proposition}\label{p.compactify-lifted}
The rescaled lifted flow $(\cL\varphi_t^*)_{t\in \RR}$
extends as a local $C^k$ fibered flow $(\widehat{\cL\varphi^*_t})_{t\in \RR}$ on $\widehat{TM}$.
The rescaled tangent flow $(\Phi_t^*)_{t\in \RR}$ extends as a linear flow
$(\widehat{\Phi_t^*})_{t\in \RR}$ on $\widehat{TM}$.

Moreover, there exists $\beta>0$ such that, for each $t\in [0,1]$, $\sigma\in {\rm Sing}(X)$ and $x=(0,u)\in p^{-1}(\sigma)$,
on the ball $B(0_x,\beta)\subset \widehat {T_{x}M}$ the map $\widehat{\cL\varphi_t^*}$
writes as:
\begin{equation}\label{e.extend}
y\mapsto \frac{\|DX(\sigma)\cdot u\|}{\|DX(\sigma)\circ D\varphi_t(\sigma)\cdot u\|}D\varphi_t(\sigma)\cdot y.
\end{equation}
\end{Proposition}
\medskip

Before proving the proposition, one first shows:
\begin{Lemma}\label{l.extend}
The function $(x,t)\mapsto \frac{\|X(x)\|}{\|X(\varphi_t(x))\|}$ on $(M\setminus \text{Sing}(X))\times \RR$
extends as a positive $C^{k-1}$ function $\widehat M\times \RR\to \RR_+$
which is equal to $\frac{\|DX(\sigma)\cdot u\|}{\|DX(\sigma)\circ D\varphi_t(\sigma)\cdot u\|}$
when  $x=(0,u)\in p^{-1}(\sigma)$.

The map from $TM|_{M\setminus \text{Sing}(X)}$
into itself which sends $y\in T_xM$ to $\|X(x)\|\cdot y$, extends as a continuous map of $\widehat {TM}$
which vanishes on the set $p^{-1}(\text{Sing}(X))$ and is $C^{k-1}$ fibered.
\end{Lemma}
\begin{proof} From Proposition~\ref{p.extended-field},
in the local chart of $0=\sigma\in{\rm Sing}(X)$,
the map 
$x\mapsto \frac{\|X(x)\|}{\|\exp^{-1}_\sigma(x)\|}$
extends as a $C^{k-1}$ function which coincides at $u\in PT_\sigma M$ with $\|DX(\sigma)\cdot u\|$
and does not vanish.
We also extend the map $(x,t)\mapsto \|\exp_\sigma^{-1}(\varphi_t(x))\|/\|\exp_\sigma^{-1}(x)\|$ as a $C^{k-1}$ map
on $\widehat M\times \RR$ which coincides with $\|D\varphi_t(\sigma)\cdot u\|$ when $x=(0,u)$.
The proof is similar to the proof of Proposition~\ref{p.extended-field}.
This implies the first part of the lemma.

For the second part, one considers the product of the $C^{k-1}$ function $x\mapsto \frac{\|X(x)\|}{\|\exp_\sigma^{-1}(x)\|}$
with the $C^\infty$ fibered map which extends $y\mapsto \|\exp_\sigma^{-1}(x)\|\cdot y$.
\end{proof}
\medskip

\begin{proof}[Proof of Proposition~\ref{p.compactify-lifted}]
In local coordinates, the local flow $(\cL \varphi_t^*)_{t\in \RR}$ in $T_xM$ acts like:
$$\cL \varphi_t^*(y)=\|X(\varphi_t(x))\|^{-1}\left(\varphi_t(x+\|X(x)\|\cdot y)-\varphi_t(x)\right)$$
\begin{equation}\label{e.tangent-extension}
=\frac{\|X(x)\|}{\|X(\varphi_t(x))\|}
\int_0^1 D\varphi_t\left(x+r\|X(x)\|\cdot y\right)\cdot y\;dr.
\end{equation}
By Lemma~\ref{l.extend},
$(x,t)\mapsto \frac{\|X(x)\|}{\|X(\varphi_t(x))\|}$ and $(x,y)\mapsto \|X(x)\|\cdot y$ extend
as a continuous maps on $\widehat M\times \RR$ and
$\widehat{TM}$ respectively; hence
$(\cL \varphi_t^*)_{t\in \RR}$ extends continuously at $x=(0,u)\in p^{-1}(\sigma)$
as in~\eqref{e.extend}.
The extended flow is $C^k$ along each fiber. Moreover,
\eqref{e.tangent-extension} implies that it is $C^{k-1}$ fibered.
For $x\in M\setminus {\rm Sing}(X)$, the $k^{th}$ derivative along the fibers is equal to
$$\frac{\|X(x)\|^k}{\|X(\varphi_t(x))\|}D^k\varphi_t(x+\|X(x)\|\cdot y).$$
This converges to $\frac{\|DX(\sigma)\cdot u\|}{\|DX(\sigma)\circ D\varphi_t(\sigma)\cdot u\|}D\varphi_t(\sigma)$
when $k=1$ and to $0$ for $k>1$. 
The extended rescaled lifted flow is thus a local $C^{k}$ fibered flow defined on a uniform neighborhood of the $0$-section.

From Lemma~\ref{l.extend}, the rescaled linear flow $(\Phi_t^*)_{t\in \RR}$
extends to $\widehat {TM}$ and coincides at $x=(0,u)\in p^{-1}(\sigma)$ with
the map defined by~\eqref{e.extend}.
From~\eqref{e.tangent-extension},
it coincides also with the flow tangent to $(\widehat{\cL \varphi_t^*})_{t\in \RR}$ at the $0$-section.

In order to define $\widehat{\cL \varphi_t^*}$ on the whole bundle $\widehat {TM}$
(and get a fibered flow as in Definition~\ref{d.local-flow}), one first glues each diffeomorphism $\widehat{\cL \varphi_t^*}$
for $t\in [0,1]$ on a small uniform neighborhood of $0$
with the linear map $D\varphi_t^*$ outside a neighborhood of $0$
in such a way that $\widehat{\cL \varphi_0^*}=\id$.
One then defines $\widehat{\cL \varphi_t^*}$ for other times by:
$$\widehat{\cL \varphi_{-t}^*}=\left( \widehat{\cL \varphi_t^*}\right)^{-1} \text{ for } t>0,$$
$$\widehat{\cL \varphi_{n+t}^*}=\widehat{\cL \varphi_{t}^*}\circ \widehat{\cL \varphi_{1}^*} \circ \dots \circ \widehat{\cL \varphi_{1}^*}
\quad \text{ ($n+1$ terms), \quad for } t\in [0,1] \text{ and } n\in \NN.$$
\end{proof}
\medskip

In a same way we compactify the rescaled fiber-preserving lifted flow $(\cL_0\varphi_t^*)$.

\begin{Proposition}
The rescaled fiber-preserving lifted flow $(\cL_0\varphi_t^*)_{t\in \RR}$ extends as a local $C^{k}$ fibered flow
on $\widehat {TM}$.
More precisely, for each $x\in \widehat{M}$, it defines a $C^k$ map
$\widehat{\cL_0\varphi^*}$
from $\RR\times\widehat {TM}_{x}$ to $\widehat{TM}_{x} $
which depends continuously on $x$ for the $C^k$ topology.

Moreover there exists $\beta>0$ such that, for each $t\in [0,1]$, $\sigma\in {\rm Sing}(X)$ and $x=(0,u)\in p^{-1}(\sigma)$,
on the ball $B(0_x,\beta)\subset \widehat{TM}_{x}$ the map $\widehat{\cL_0\varphi_t^*}$
has the form:
\begin{equation}\label{e.ext-L0}
y\mapsto D\varphi_t(\sigma)\cdot y +\frac{D\varphi_t(\sigma)\cdot u-u}{\|DX(\sigma)\cdot u\|}.
\end{equation}
\end{Proposition}
\begin{proof}
In the local coordinates the flow acts on $B(0_x,\beta)\subset T_xM$ as:
\begin{align}\label{e.L0}
\cL_0\varphi_t^*(y)&=\|X(x)\|^{-1}\left(\varphi_t(x+\|X(x)\|\cdot y)-x \right)\\
&= \int_0^1D\varphi_t(x+r\|X(x)\|\cdot y)\cdot y\;dr \; +\;\frac{\varphi_t(x)-x}{\|X(x)\|}.\label{e.extend-L0}
\end{align}
Arguing as in  Proposition~\ref{p.extended-field} and Lemma~\ref{l.extend},
{for each $t$, the map $x\mapsto \frac{\varphi_t(x)-x}{\|X(x)\|}$}
with $x\neq \sigma$ close to $\sigma$ extends for the $C^0$-topology
 by $\|DX(\sigma)\cdot u\|^{-1} (D\varphi_t(\sigma)\cdot u-u)$ at points $(0,u)\in PT_\sigma M$.
Since $X$ is $C^k$, these maps are all $C^k$ and depends continuously
with $x$ for the $C^k$ topology.

As before, the integral 
$\int_0^1D\varphi_t(x+r\|X(x)\|\cdot y)\cdot y\;dr$ extends as $D\varphi_t(\sigma)\cdot y$
at $p^{-1}(\sigma)$.
For each $x$, the map $(t,y)\mapsto\int_0^1D\varphi_t(x+r\|X(x)\|\cdot y)\cdot y\;dr$
is $C^k$ (this is checked on the formulas, considering separately the cases $x\in M\setminus  {\rm Sing}(X)$
and $x\in p^{-1}( {\rm Sing}(X))$. Since $X$ is $C^k$, this map depends continuously on $x$
for the $C^{k-1}$-topology.
The $k^\text{th}$ derivative with respect to $y$ is continuous in $x$, for the same reason as in the proof
of Proposition~\ref{p.compactify-lifted}.
For $x\in M\setminus  {\rm Sing}(X)$, the derivative with respect to $t$
of the map above is $(t,y)\mapsto\int_0^1DX(\varphi_t(x+r\|X(x)\|\cdot y))\cdot y\;dr$
and it converges as $x\to \sigma$ towards $(x,y)\mapsto DX(\varphi_t(\sigma))\cdot y$
for the $C^{k-1}$-topology (again using that $X$ is $C^k$).
Hence the first term of~\eqref{e.extend-L0} is a $C^k$ function of $(t,y)$ which depends continuously on $x$
for the $C^k$ topology.

As in Proposition~\ref{p.compactify-lifted}, this proves that $(\cL_0\varphi_t^*)_{t\in \RR}$ extends as a local $C^{k}$ flow having the announced properties.
\end{proof}

\subsection{The \emph{extended rescaled non-linear Poincar\'e flow}}
We also obtain a compactification of the rescaled non-linear Poincar\'e flow $(\psi_t^*)_{t\in \RR}$
(and of the rescaled linear Poincar\'e flow $(\Psi^*_t)_{t\in \RR}$).
This implies the first part of our Main Theorem.
\begin{Theorem}\label{t.compactified2}
If $X$ is $C^k$, $k\geq 1$, and if
$DX(\sigma)$ is invertible at each singularity,
the rescaled non-linear Poincar\'e flow $(\psi_t^*)_{t\in \RR}$ extends as a local $C^k$ fibered flow
$(\widehat{\psi_t^*})_{t\in \RR}$ on a neighborhood of the $0$-section in $\widehat {\cN}$.
Moreover, there exists $\beta'>0$ such that for each $t\in [0,1]$, $\sigma\in {\rm Sing}(X)$ and
$x=(0,u)\in p^{-1}(\sigma)$,
on the ball $B(0_x,\beta')\subset {\widehat {{\cN}_{x}}}$ the map $\widehat{\psi^*_t}$
writes as:
$$y\mapsto \;\frac{\|DX(\sigma)\cdot u\|}{\|DX(\sigma)\circ D\varphi_t(\sigma)\cdot u\|}D\varphi_{t+\tau_x}(\sigma)\cdot y\;
+\;\frac{D\varphi_{t+\tau_x}(\sigma)\cdot u-D\varphi_{t}(\sigma)\cdot u}{\|DX(\sigma)\circ D\varphi_t(\sigma)\cdot u\|},$$
where $\tau_x$ is a $C^{k}$ function of $(t,y)\in [0,1]\times B(0,\beta')$
which depends continuously on $x$ for the $C^k$ topology,
such that $\tau_x(t,0_x)=0$.
\smallskip

As a consequence, the rescaled linear Poincar\'e flow $(\Psi_t^*)_{t\in \RR}$ extends as a continuous
linear flow $(\widehat{\Psi_t^*})_{t\in \RR}$: for each $x\in\widehat M$, each $t\in\RR$ and each $v\in\widehat {\cN_{x}}$
the image $\widehat{\Psi_t^*}\cdot v$ coincides with the normal projection of $D\varphi_t^*\cdot v\in {\widehat {TM}_{x'} }$ on $\widehat{\cN_{x'}}$, where $x'=\widehat\varphi_t({x})$.
\end{Theorem}
\begin{proof}
For each singularity $\sigma$, we work with the local coordinates $(-\varepsilon, \varepsilon)\times T^1_\sigma M$
and prove that the rescaled non-linear Poincar\'e flow extends as a local $C^k$ fibered flow.
Since the rescaled non-linear Poincar\'e flow is invariant by the symmetry $(s,u)\mapsto (-s,-u)$,
this implies the result above a neighborhood of $p^{-1}(\sigma)$ in $\widehat M$,
and hence above the whole manifold $\widehat M$.

The image of $y\in B(0,\beta)\subset \cN_x$ by
the time $t$ of the rescaled non-linear Poincar\'e flow is the unique point of the curve in $T_{\varphi_t(x)}M$
$$\tau\mapsto \cL_0\varphi_\tau^*\circ \cL\varphi_t^*(y)$$
which belongs to $\cN_{\varphi_t(x)}$.
In the local coordinates $x=(s,u)\in (-\varepsilon, \varepsilon)\times T^1_\sigma M$ near $\sigma$,
it corresponds to the unique value
$\tau$ such that the following function vanishes:
$$\Theta(x,t,y,\tau)=\bigg\langle \cL_0\varphi_\tau^*\circ \cL\varphi_t^*(y)\;,\;
\frac{X(\varphi_t(x))}{\|X(\varphi_t(x))\|}\bigg\rangle.$$

From the previous propositions the map $(y,\tau)\mapsto \Theta(x,t,y,\tau)$ is $C^k$ and depends continuously
on $(x,t)$ for the $C^{k}$-topology and is defined at any $x\in\widehat M$.
When $x=(0,u)\in p^{-1}(\sigma)$, the first part of the scalar product in $\Theta$ is given by the two propositions above.
According to Proposition~\ref{p.extended-field} the second part
$\widehat X_1(x)$
becomes to
$$\frac{DX(\sigma)\circ D\varphi_t(\sigma)\cdot u}{\|DX(\sigma)\circ D\varphi_t(\sigma)\cdot u\|}.$$

\begin{Claim}
For any $t\in[0,1]$ and for $x\in p^{-1}(\sigma)$, the derivative $\frac{\partial \Theta}{\partial \tau}\big|_{\tau=0}(x,t,0)$ is non-zero. 
\end{Claim}
\begin{proof}
Indeed, by the previous proposition this is equivalent to
$$\bigg\langle\frac{\partial}{\partial\tau}\bigg|_{\tau=0}(D\varphi_\tau(\sigma)\cdot v)\;,\;
\frac{DX(\sigma)\circ D\varphi_t(\sigma)\cdot u}{\|DX(\sigma)\circ D\varphi_t(\sigma)\cdot u\|}\bigg\rangle\neq 0,$$
where $v=\widehat \varphi_t\cdot u={D\varphi_t(\sigma)\cdot u}/{\|D\varphi_t(\sigma)\cdot u\|}$.
Thus the claim follows from the fact that
$\|DX(\sigma)\circ D\varphi_t(\sigma)\cdot u\|$ does not vanish.
\end{proof}

By the implicit function theorem and compactness, there exists $\beta'>0$ and, for each $x$, a $C^{k}$ map
$(y,t)\mapsto\tau_x(t,y)$ which depends continuously on $x$ for the $C^k$ topology such that
$$\Theta(x,t,y,\tau_x(t,y))=0,$$
for each $x\in\widehat M$ close to $PT_\sigma M$, each $t\in[0,1]$ and each $y\in B(0,\beta')$.
The rescaled non-linear Poincar\'e flow is thus locally given by the composition:
\begin{equation}\label{e.rescaledpoincare}
(x,t,y)\mapsto \cL_0\varphi_{\tau_x(t,y)}^*\circ \cL\varphi_t^*(y),
\end{equation}
which extends as a $C^{k}$ fibered flow. The formula at $x=(0,u)\in p^{-1}(\sigma)$
is obtained from the expressions in the previous propositions.
\medskip

We now compute the rescaled linear Poincar\'e flow as the tangent map to the rescaled non-linear Poincar\'e flow along the $0$-section. We fix $x\in\widehat M$ and its image
$x'=\widehat \varphi_t(x)$. We take $y\in \cN_{x}$
and its image $y'\in \widehat {TM}_{x'}$ by $\cL\varphi^*_t$.
Working in the local coordinates $(-\varepsilon,\varepsilon)\times T^1_\sigma M$
and using Proposition~\ref{p.extended-field} and formulas~\eqref{e.ext-L0} and~\eqref{e.L0} we get
$$\frac{\partial}{\partial \tau}\bigg|_{\tau=0}\cL_0\varphi^*_\tau(0_{x'})=\widehat X_1(x').$$
{Note that }$\tau(x',t,0)=0$. We have:
\begin{equation*}
\begin{split}
D\psi^*_t(0)&=
\left( \frac{\partial}{\partial y'}\bigg|_{y'=0}  \cL_0\varphi^*_0(y') \right)\circ\left( \frac{\partial}{\partial y}\bigg|_{y=0}\cL\varphi_t^*(y)\right)+
\frac{\partial\tau}{\partial y}\bigg|_{y=0}\left( \frac{\partial}{\partial \tau}\bigg|_{\tau=0}\cL_0\varphi^*_\tau(0_{x'})\right)\\
&= D\varphi_t^*(x)+\frac{\partial\tau}{\partial y}\bigg|_{y=0}\widehat X_1(x').
\end{split}
\end{equation*}

On the other hand from~\eqref{e.rescaledpoincare} and the definitions of $\Theta,\tau$ we have
$$\left< D\psi^*_t(0),\widehat X_1(x')\right>=0,$$
hence $\Psi^*_t=D\psi^*_t(0)$ coincides with the normal projection of $D\varphi_t^*(x)$
on the linear sub-space of $\widehat {TM}_{x'}$ orthogonal to $\widehat X_1(x')$, which is $\widehat {\cN}_{x'}$.
\end{proof}

\subsection{Proof of The Main Theorem (without the identifications)}\label{ss.first-part}
Let $\Lambda$ be the compact invariant set in the assumptions of the Main Theorem. Recall the blowup $p:~\widehat M\to M$. We define $\widehat\Lambda$ as the closure of $p^{-1}(\Lambda\setminus{\rm Sing}(X)))$ in $\widehat M$. Since the flow $\varphi$ on $M$ induces a flow $\widehat\varphi$ on $\widehat M$, we know that the restriction of $\varphi$ to $\Lambda\setminus{\rm Sing}(X)$ embeds in $({\widehat\Lambda},{\widehat\varphi})$ through the map $i=p^{-1}$.
The metric on the bundle $\widehat{TM}$ over $\widehat M$ is the pull back metric of $TM$
and the map $I$ induces an isometric isomorphism between $\cN|_{\Lambda\setminus{\rm Sing}(\Lambda)}$
and $\widehat{\cN}|_{p^{-1}(\Lambda\setminus{\rm Sing}(\Lambda))}$

Finally Theorem~\ref{t.compactified2} shows that the $C^k$ local fibered flow $(\widehat {\Psi^*_t})_{t\in \RR}$ is conjugated by $I$ near the zero-section to the rescaled non-linear Poincar\'e flow $(\Psi^*_t)_{t\in \RR}$.
\qed

\subsection{The application to the vector field over $M$}
To summarize what we have proved, one has the following theorem on $M$.
\begin{Theorem}\label{t.non-degeneracy}
Assume that $X$ is a $C^k$ vector field over $M$, $k\geq 1$, such that $DX(\sigma)$ is invertible for any singularity $\sigma$.  Then given $t,\varepsilon>0$, there is $\delta,\beta>0$ such that for any regular points $x,y\in M$, if $\RR.X(x)$ and $\RR.X(y)$ are $\delta$-close, then $\psi_{t,x}^*|_{B(0_x,\beta)}$ and $\psi_{t,y}^*|_{B(0_y,\beta)}$ are $\varepsilon$-close in the $C^k$ topology.
\end{Theorem}
Indeed if for any regular $x,y\in M$, the directions $\RR.X(x)$ and $\RR.X(y)$ are $\delta$-close,
then the corresponding points in the blownup manifold $\widehat M$ are $C\delta$ close for some uniform $C$.

\section{Identifications}\label{s.identifications}
Identifications (see Definition~\ref{Def:identification}) provide a structure
on the local fibered flows which formalizes properties
satisfied by the rescaled non-linear Poincar\'e flow.
We introduce here an additional property (the compatibility) which will
be checked on the extended rescaled non-linear Poincar\'e flow in Section~\ref{Existence}.
We also discuss some consequences.

In the whole section we consider a $C^k$ local fibered flow
$(\psi_t)_{t\in\RR}$ on a vector bundle $\cN\to K$
that fibers over a topological flow $(\varphi_t)_{t\in \RR}$
on the compact metric space $K$. 

\subsection{Compatible identifications}
We have introduced in Definition~\ref{Def:identification}
the notion of $C^k$ identifications $\mathfrak{h}$ on an open set $U\subset K$:
these are continuous families of $C^k$ diffeomorphisms $\mathfrak{h}_{y,x}\colon \cN_y\to \cN_x$ indexed by pairs of points $x,y \in K$, which satisfy
the relation~\eqref{e.identification} on balls $B(0_x,\beta_0)\subset \cN_x$
when $x,y,z$ are $r_0$-close points in $U$.
In particular $\mathfrak{h}_{x,x}$ coincides with the identity on the ball $B(0_x,\beta_0)\subset \cN_x$.
\medskip

\noindent
{\bf Notations.} -- In order to simplify the presentation, we will sometimes denote $\mathfrak{h}_{y,x}$
by $\mathfrak{h}_x$. Also the projection $\mathfrak{h}_{y,x}(0)=\mathfrak{h}_x(0_y)$ of $0_y\in \cN_y$ on $\cN_x$ will be denoted
by $\mathfrak{h}_x(y)$.
\medskip

We will denote by ${\rm Lip}$ be the set of orientation-preserving bi-Lipschitz homeomorphisms $\theta$ of $\RR$
(and by ${\rm Lip}_{1+\rho}$ the set of maps in ${\rm Lip}$ whose Lipschitz constant is smaller than $1+\rho$).

\begin{Definition}\label{d.compatible}
The identification $\mathfrak{h}$ on $U$ is \emph{compatible} with the local fibered flow $(\psi_t)$ if:
\begin{itemize}

\item[--] \emph{No small period.}
For any $\varkappa>0$, there is $r>0$ such that for any $x\in \overline{U}$
and $t\in [-2,2]$ with $d(x,\varphi_t(x))<r$, then we have $|t|< \varkappa$.

\item[--] \emph{Local injectivity.} For any $\delta>0$, there exists $\beta>0$ such that for any $x,y\in U$:\\
if $d(x,y)<r_0$ and $\|\mathfrak{h}_{x}(y)\|\leq \beta$, then $d(\varphi_t(y),x)\leq \delta$ for some $t\in [-1/4,1/4]$.

\item[--] \emph{Local invariance.} For any $x,y\in U$
and $t\in [-2,2]$ such that $y$ and $\varphi_t(y)$ are $r_0$-close to $x$,
and for any $u\in B(0,\beta_0)\subset \cN_{y}$, we have
$$\mathfrak{h}_{x}\circ \psi_t(u)=\mathfrak{h}_{x}(u).$$

\item[--] \emph{Global invariance.}
For any $\delta,\rho>0$, there exist $r,\beta>0$ with the following properties:

{\rm 1.} For any $y,y'\in U$ with $d(y,y')<r$,
for any $u\in  \cN_y$, $u'\in \cN_{y'}$ with $\mathfrak{h}_y(u')=u$,
and for any intervals $I,I'\subset \RR$ containing $0$ and satisfying
$$\|\psi_t(u)\|<\beta\text{ and } \|\psi_{t'}(u')\|<\beta\text{ for any } t\in I \text{ and any }
t'\in I',$$
there is $\theta\in {\rm Lip}_{1+\rho}$ such that $\theta(0)=0$ and
 $d(\varphi_t(y),\varphi_{\theta(t)}(y'))<\delta$ for any $t\in I\cap \theta^{-1}(I')$.\\
{\rm 2.} In addition, if some $v\in \cN_y$ satisfies
$[\forall t\in I\cap \theta^{-1}(I'),\; \|\psi_t(v)\|<\beta]$, then:
\begin{itemize}
\item[\rm 2.i)] $v':=\mathfrak{h}_{y'}(v)$ satisfies
$[\forall t\in I\cap \theta^{-1}(I'),\; \|\psi_{\theta(t)}(v')\|<\delta]$;
\item[\rm 2.ii)] if $\varphi_t(y)\in U$
for some $t\in I\cap \theta^{-1}(I')$, then $\mathfrak{h}_{\varphi_t(y)}\circ \psi_{\theta(t)}(v')=\psi_t(v)$.
\end{itemize}
\end{itemize}
\end{Definition}

\begin{Remarks-numbered}\label{r.identification}\rm
a) These definitions are still satisfied if one reduces
$r_0$ or $\beta_0$. Their values may be reduced in the following sections in order to
satisfy additional properties.
\medskip

\noindent
b) One can rescale the time and keep a compatible identification: the flow $s\mapsto \varphi_{C\cdot s}$ for $C>1$
still satisfies the definitions above, maybe after reducing the constant $r_0$.

The main point to check is that the time in the Local injectivity can still
be chosen in $[-1/4,1/4]$. Indeed, this is ensured by the ``No small period'' Property applied with $\varkappa=1/(4C)$: if $r_0$ is chosen smaller
and if $d(\varphi_t(y),x),d(y,x)$ are both less than $r_0$ for $t\in [-1/4,1/4]$, then $|t|$ is smaller than $\varkappa$.
Now the time in the Local injectivity property belongs to $[-\varkappa,\varkappa]$ for the initial flow, hence to $[-1/4,1/4]$ for the time-rescaled flow.
\medskip

\noindent
c) Applying a rescaling as discussed in the item (b),
one can replace the interval $[-2,2]$ in the ``No small period'' Property and in the Local invariance
by any larger interval $[-L,L]$; one can replace the interval $[-1/4,1/4]$ in the local injectivity
by any small interval $[-\eta,\eta]$.
\medskip

\noindent
d) The ``No small period'' Property (which does not involve the projections $\mathfrak{h}_x$) is equivalent to the non-existence of periodic orbits of period $\leq 2$
which intersect $\overline U$.
In particular, by reducing $r_0$, one can assume the following property:
\smallskip

\emph{For any $x\in U$ and any $t\in [1,2]$, we have $d(x,\varphi_t(x))\ge r_0$.}
\medskip

\noindent
e) For $x\in U$, the Local injectivity prevents the existence of  $y\in U$ that is $r_0$-close to $x$,
is different from $\varphi_t(x)$ for any $t\in [-1/4,1/4]$, and
such that $\mathfrak{h}_x(y)=0_x$.
In particular:
\smallskip

\emph{If $x,\varphi_t(x)\in U$ satisfy $d(x,\varphi_t(x))<r_0$, $t\not\in (-1/2,1/2)$ and
$\mathfrak{h}_x(\varphi_t(x))=0_x$, then $x$ is periodic.}\hspace{-1cm}\mbox{}
\medskip

\noindent
f) The Global invariance says that when two orbits $(\psi_t(u))_{t\in I}$ and
$(\psi_t(u'))_{t\in I'}$ of the local fibered flow are close to
the $0$ section of $\cN$ and have two points which are identified by $\mathfrak{h}$,
then they are associated to orbits of the flow $(\varphi)_{t\in \RR}$ that are close
(up to a reparametrization $\theta$).
In this case, pieces of orbit of $(\psi_t)_{t\in \RR}$ close to the zero-section
above the first $\varphi$-orbit can be projected to an orbit of $(\psi_t)_{t\in \RR}$ above the second $\varphi$-orbit.
\medskip

\noindent
g) The Global invariance can be applied to pairs of points $y,y'$ where the condition
$d(y,y')<r$ has been replaced by a weaker one $d(y,y')<r_0$. More precisely, the following holds:
\medskip

\emph{For any $\delta,\rho>0$, there is $\beta>0$ as follows:
for any $y,y'\in U$ with $d(y,y')<r_0$, for any $u\in  \cN_y$, $u'\in \cN_{y'}$ with $\mathfrak{h}_y(u')=u$
and for any the intervals $I,I'\subset \RR$ containing $0$ and satisfying:
$$\|\psi_t(u)\|<\beta\text{ and } \|\psi_{t'}(u')\|<\beta\text{ for any } t\in I \text{ and any }
s'\in I',$$
then there is $\theta\in {\rm Lip}_{1+\rho}$ such that $|\theta(0)|\leq 1/4$ and $d(\varphi_t(y),\varphi_{\theta(t)}(y'))<\delta$ for any $t\in I\cap \theta^{-1}(I')$.}
\medskip

Indeed provided that $\beta>0$ has been chosen small enough,
one can apply the Local injectivity and the Local invariance
in order to replace $y'$ and $u'$ by $y''=\varphi_s(y')$ and $u''=\psi_s(u')$
for some $s\in [-1/4,1/4]$ such that $d(y,y'')<r$.
The assumptions for the Global invariance then are satisfied by $y,y''$ and $u,u''$.
It gives a  $\theta\in {\rm Lip}_{1+\rho}$
satisfying $d(\varphi_t(y),\varphi_{\theta(t)}(y'))<\delta$ for $t\in I\cap \theta^{-1}(I')$
but the condition $\theta(0)=0$ has been replaced by
$\theta(0)=s$; in particular $|\theta(0)|<1/4$.
\end{Remarks-numbered}

\subsection{No shear inside orbits}
The next property states that along reparametrizations of orbits of the flow $(\varphi_t)_{t\in\RR}$,
the time cannot fluctuate widely. This will be useful for applications announced in~\cite{CY1}.

\begin{Proposition}\label{p.no-shear}
Let us assume that there exists a $C^k$ identification on an open set $U\subset K$
which is compatible with the local fibered flow $(\psi_t)_{t\in\RR}$
and that the set $\Delta:=\{x| \forall |t|\leq \tfrac 1 2,\; \varphi_t(x)\notin U\}$
is disjoint from $\overline U$. If $r_0>0$ is small, then
for any $x\in U$, any increasing homeomorphism $\theta$ of $\RR$, any closed interval $I$ containing $0$ satisfying $[d(\varphi_t(x),\varphi_{\theta(t)}(x))\leq r_0,\;\forall t\in I]$ and $\varphi_{\theta(0)}(x)\in U$,
the following holds:
\begin{itemize}
\item[--] If $\theta(0)> 1/2$, then $\theta(t)> t+2$, $\forall t\in I$ when $\varphi_t(x), \varphi_{\theta(t)}(x)\in U$;
\item[--] If $\theta(0)\in [-2,2]$, then $\theta(t) \in [t-1/2,t+1/2]$, $\forall t\in I$ when $\varphi_t(x),\varphi_{\theta(t)}(x)\in U$;
\item[--] If $\theta(0)< -1/2$, then $\theta(t)<t-2$, $\forall t\in I$ when $\varphi_t(x),\varphi_{\theta(t)}(x)\in U$.
\end{itemize}
\end{Proposition}
\begin{proof}
Let us take $r_0, \varepsilon>0$ small so that $\varphi_s(\overline U)$ and $\Delta$ are at distance larger than $r_0$ when $|s|\leq \varepsilon$. One also assume that:
\begin{itemize}
\item[a)] any piece of orbit
$\{\varphi_s(y),s\in [0,b]\}\subset K\setminus U$, with $y,\varphi_b(y)$ in the $r_0$-neighborhood of $\Delta$ and $b\leq 1/2$, is disjoint from the $r_0$-neighborhood of $\overline U$;
\item[b)] if $\varphi_s(y)$ is $r_0$-close to $y\in\overline U$ for $|s|\leq 2$, then $|s|\leq \varepsilon$.
\end{itemize}
The first condition is satisfied by small $r_0>0$ since
otherwise letting $r_0\to 0$ one would construct $y,\varphi_b(y)\in \Delta$
and $\varphi_{s}(y)\in \overline U$ where $0\leq s< b\leq 1/2$, contradicting the definition of $\Delta$.
The second condition is  a consequence of the ``No small period'' Property.

\begin{Lemma}
\begin{enumerate}
\item[(i)] If $\theta(0)\geq -2$ then $\theta(t)\geq t-1/2$ for any $t\in I$ satisfying $\varphi_t(x),\varphi_{\theta(t)}(x)\in U$.
\item[(ii)] If $\theta(0)\leq 2$ then $\theta(t)\leq t+1/2$ for any $t\in I$ satisfying $\varphi_t(x),\varphi_{\theta(t)}(x)\in U$.
\end{enumerate}
\end{Lemma}
\begin{proof}
The case $t=0$ is a consequence of the ``No small period'' Property.
We now deal with the assertion (i) in the case {that} $t$ is positive.

Assume by contradiction that there is $t_0>0$ in $I$ such that $\varphi_{t_0}(x)\in U$, but $\theta(t_0)<t_0-1/2$.
Since $\theta(0)\geq -1/2$, one can consider the maximal interval $J$ containing $0$
of parameters $t\in I$ satisfying either $\varphi_t(x)\not\in \overline U$,
or $\theta(t)\geq t-1/2$.

Let $t_1$ be the largest time in $[0,t_0]\cap J$
satisfying $\varphi_{t_1}(x)\in \overline U$. It exists since
$x=\varphi_0(x)\in \overline U$ and
if $(t_k)$ is an increasing sequence in $[0,t_0]\cap J$ satisfying
$\varphi_{t_k}(x)\in \overline U$, then we have $\theta(t_k)\geq t_k-1/2$.
So the limit $\overline t$ satisfies $\theta(\overline t)\geq \overline t-1/2$ (and belongs to $J$)
and $\varphi_{\overline t}(x)\in \overline U$.

By construction $\theta(t_1)\geq t_1-1/2$, hence $t_1<t_0$.
We claim that $\theta(t_1)\geq t_1-\varepsilon$.
Indeed, one can reduce to consider the case $\theta(t_1)\leq t_1+2$
(otherwise the property holds trivially). Then, Property (b) gives $|\theta(t_1)-t_1|\leq \varepsilon$
which implies the claim.

For $s>t_1$ close to $t_1$, we thus have $s\in J$. Since $t_1$ is maximal
we also get  $\varphi_{s}(x)\not\in \overline U$. Since $\varphi_{t_0}(x)\in U$,
there exists a minimal $t_2\in (t_1,t_0]$ such that $\varphi_{t_2}(x)$ belongs to the boundary of $U$.
In particular $\varphi_{\theta(t_2)}(x)$ is $r_0$-close to the boundary of $U$.
Note that $[t_1,t_2)\subset J$. By maximality of $t_1$ one has $\theta(t_2)<t_2-1/2$.

The ``No small period'' Property implies $\theta(t_2)<t_2-2$.
In particular,
$$t_2>\theta(t_2)+2>\theta(t_1)+2\geq t_1+2-\varepsilon.$$
Hence the interval $[t_1+1/2,t_2-1/2]$ is non-empty. By construction
any time $s$ in this interval satisfies $\varphi_{s+t}(x)\notin U$ for
all $|t|\leq 1/2$. One deduces $\varphi_{s}(x)\in \Delta$ for all $s\in [t_1+1/2,t_2-1/2]$.

Since $\varphi_{\theta(t_1+1/2)}(x)$ is $r_0$-close to $\varphi_{t_1+1/2}(x)\in \Delta$,
and $\varphi_{t_1}(x)\in\overline U$,
one has $|\theta(t_1+1/2)-t_1|> \varepsilon$ (by our choice of $\varepsilon$).
Since $\theta(t_1)\geq t_1-\varepsilon$, this gives
$\theta(t_1+1/2)> t_1+\varepsilon$.

We then distinguish two cases:
\begin{itemize}
\item[--] When $\theta(t_2)\in [t_1+1/2,t_2-1/2]$, then $\varphi_{\theta(t_2)}(x)\in \Delta$
is $r_0$ far from $\overline U$.
This is a contradiction since we have proved before that $\varphi_{\theta(t_2)}(x)$ is $r_0$-close to the boundary of $U$. 

\item[--] When $\theta(t_2)< t_1+1/2$, then $t_1+\varepsilon<\theta(t_1+1/2)<\theta(t_2)< t_1+1/2$.
In particular the interval $[\theta(t_1+1/2),t_1+1/2]$ has length smaller than $1/2$. Since
$\varphi_{\theta(t_1+1/2)}(x)$ belongs to the $r_0$-neighborhood of $\Delta$
and since $\varphi_{t_1+1/2}(x)\in \Delta$, the property (a) implies that
$\varphi_{\theta(t_2)}(x)$ is disjoint from the $r_0$-neighborhood of $\overline U$. This also gives a contradiction.
\end{itemize}
In both cases we found a contradiction. Hence the conclusion (i) of the lemma holds when $t$ is positive.
A symmetric argument gives the conclusion (ii) when $t$ is negative.
\medskip

Let us consider the parametrization of the orbit of $x'=\varphi_{\theta(0)}(x)\in U$
by the increasing homeomorphism $\theta'\colon s\mapsto \theta^{-1}(s+\theta(0))-\theta(0)$.
Note that $\varphi_{\theta'(0)}(x')=\varphi_{-\theta(0)}(x')=x$ belongs to $U$.
For any $s\in I':=\theta(I)-\theta(0)$,
by setting $t=\theta^{-1}(s+\theta(0))$ we get
$$d(\varphi_s(x'),\varphi_{\theta'(s)}(x'))=d(\varphi_{\theta(t)}(x),\varphi_t(x))<r_0.$$

Let us assume that $\theta(0)\geq -2$ and that some
negative $t\in I$ satisfies $\theta(t)\geq t-1/2$ and $\varphi_t(x)\in U$ as in assertion (i).
We then have $\theta'(0)=-\theta(0)\leq 2$.
We may apply assertion (ii) for the negative time $s=\theta(t)-\theta(0)$, for $x',\theta',I'$.
Indeed we get
$$\varphi_{\theta'(s)}(x')=\varphi_t(x)
\text{ and } \varphi_s(x')=\varphi_{\theta(t)}(x).$$
We thus conclude $\theta'(s)\leq s+1/2$.
The definitions of $s$ and $\theta'$ give $t\leq \theta(t)+1/2$.
The assertion (i) is now proved in all cases. The assertion (ii) is obtained by a symmetric argument.
\end{proof}

The previous lemma immediately implies the second item of the proposition.
The first and third item also follows by replacing $x$ by $x'$ and considering the time $t'=-t$.
The proposition is thus proved.
\end{proof}

\subsection{Closing properties}
The local fibered flow $(\psi_t)_{t\in\RR}$ can detect periodic points of $(\varphi_t)_{t\in\RR}$.
\begin{Proposition}\label{p.closing0}
Let us assume that there exists a $C^k$ identification on an open set $U\subset K$
which is compatible with the local fibered flow $(\psi_t)_{t\in\RR}$
and let $\beta_0,r_0>0$ be small enough.
Let us consider:
\begin{itemize}
\item[--] $x\in U$ and $T\geq 4$ with $\varphi_T(x) \in U\cap B(x,r_0)$,
\item[--] a fixed point $p\in \cN_x$ for $\overline \psi_T:=\mathfrak{h}_x\circ \psi_T$
such that $\|\psi_t(p)\|<\beta_0$ for each $t\in [0,T]$,
\item[--]  a sequence $(y_k)_{k\in\NN}$ in a compact set of $U\cap B(x,r_0/2)$
such that $\mathfrak{h}_x(y_k)$ converges to $p$.
\end{itemize}
Then, taking a subsequence, $(y_k)_{k\in\NN}$
converges to a periodic point $y\in K$
such that $\mathfrak{h}_x(y)=p$.

Moreover, if $T'$ denotes the period of $y$, then we have
$$D\psi_{T'}(0_y)= D\mathfrak{h}_{x}(0_y)^{-1}\circ D\overline \psi_T(p) \circ D\mathfrak{h}_x(0_y).$$
\end{Proposition}
\begin{proof}
Up to extracting a subsequence, $(y_{k})$ converges to a point
$y\in U\cap B(x,r_0/2)$ such that $\mathfrak{h}_x(y)=p$.
By Global invariance, there exists $(T_k)$ satisfying
$\frac 1 2 T\leq T_k\leq 2T$ such that $\varphi_{T_k}(y_k)$
is in $B(x,r_0/2)$ and projects by $\mathfrak{h}_x$ on $\overline \psi_T(\mathfrak{h}_x(y_k))$.

In particular $(\mathfrak{h}_x\circ \varphi_{T_k}(y_k))$ converges to $p$
and $(\mathfrak{h}_y\circ \varphi_{T_k}(y_k))$ converges to $0_y$.
One deduces by Local injectivity (up to modifying $T_k$ by adding a real number in $[-1,1]$)
that $\varphi_{T_k}(y_k)$ converges to $y$.
Since $T\geq 4$, the limit value $T'$ of $T_k$ is larger than $1$ and
one deduces that $y$ is $T'$-periodic.
Since $\mathfrak{h}_x(y)=p$, by the Global invariance
$D\overline \psi_T(p)$ and $D\psi_{T'}(0_y)$ are conjugated by $D\mathfrak{h}_{x}(0_y)$.
\end{proof}

For the next statement, we consider an open set $V$ such that $K\subset U\cup V$.
We reduce $r_0>0$ so that $d(K\setminus U, K\setminus V) > r_0$.
\begin{Corollary}\label{c.closing0}
Let us assume that $\beta_0,r_0>0$ are small enough.
If $x\in K\setminus V$ has an iterate $y=\varphi_T(x)$ in $B(x,r_0)$ with $T\geq 4$ and if there exists
a subset $B\subset \cN_x$ containing $0_x$ such that:
\begin{itemize}
\item[--] $\psi_t(B)\subset B(0_{\varphi_t(x)},\beta_0)$ for any $0<t<T$,
\item[--] $\overline \psi_T:=\mathfrak{h}_x\circ \psi_T$ sends $B$ into itself,
\item[--] the sequence $\overline \psi_T^k(0_x)$ converges to a fixed point $p\in B$ of $\overline \psi$,
\end{itemize}
then the positive orbit of $x$ by $\varphi$ converges to a periodic orbit of the flow $(\varphi_t)_{t\in\RR}$.
\end{Corollary}
\begin{proof}
From our choice of $V$ and $r_0$, we have $B(x_0,r_0)\subset U$.
By Global invariance, there exists a sequence $T_k\to +\infty$
such that $y_k:=\varphi_{T_k}(x)$ projects by $\mathfrak{h}_x$ on $\overline \psi_T^k(0_x)$
and $|T_{k+1}-T_k|$ is uniformly bounded in $k$.

Since $\overline \psi_T^k(0_x)=\mathfrak{h}_x(y_k)$ converges to $p$,
we can apply the previous lemma so that a subsequence of $(y_k)_{k\in\NN}$ converges to a $T'$-periodic point $y\in K$. Note that $\mathfrak{h}_y(y_k)$ converges to $0_y$.
By Local injectivity, one can then modify each time $T_k$ by adding a real number in $[-1,1]$
so that $(y_k)_{k\in\NN}$ converges to $y$.
Since $|T_{k+1}-T_k|$ is uniformly bounded in $k$, this proves that the $\omega$-limit set of $x$ is the orbit of $y$.
\end{proof}

\section{Existence of identifications}\label{Existence}
In this section we complete the proof of the Main Theorem.

\begin{Theorem}\label{Thm:local-fibered-flow-model}
Let $X$ be a $C^k$ vector field, $k\geq 1$, on a closed manifold $M$
and $\Lambda$ be a compact set , invariant by the flow $(\varphi_t)_{t\in\RR}$
associated to $X$, such that $DX(\sigma)$ is invertible for each
$\sigma\in {\rm Sing}(X)\cap \Lambda$.
Let $(\cN,\widehat{\psi^*})$ be the local $C^k$ fibered flow over
the flow $(\widehat\Lambda,\widehat\varphi)$ and let $i,I$ be the maps
as in the first part of the statement of the Main Theorem.

Then for any open set $\widehat U\subset \widehat \Lambda$
whose closure is disjoint from $i(\Lambda\setminus{\rm Sing}(X))$, there exist $C>0$ and an identification $\widehat{\mathfrak{h}}$ on $\widehat U$ which is compatible with the flow $(\widehat{\psi_{C\cdot t}^*})_{t\in\RR}$.
\end{Theorem}
Now we give the proof of Theorem~\ref{Thm:local-fibered-flow-model}.

\subsection{Construction of identifications}
Since there is no fixed point of $(\widehat \varphi_{t})_{t\in\RR}$ in the boundary of $\widehat U$,
one can rescale the time (i.e. consider the new flow $t\mapsto {\widehat\varphi}_{C\cdot t}$ for some large $C>0$)
so that any periodic orbit which meets the closure of $\widehat U$ has period larger than $10$.
Then the ``No small period'' Property follows from Remark~\ref{r.identification}(d).
In order to simplify the presentation, we will now work with the flows after time rescaling but denote them
as $(\varphi_t)_{t\in\RR}$ and $(\psi_t)_{t\in\RR}$ without mentioning the constant $C>0$.
\medskip

Let us remember that  $p$ is the projection $\widehat\Lambda\to\Lambda$
and let us define $U:=p(\widehat U)$.
The identification is a family of diffeomorphisms $\widehat{\mathfrak{h}}_{y,x}$
defined for close points $x,y\in \widehat U$. It is enough to work with the flows
$(\varphi_t)_{t\in\RR}$ on $\Lambda\setminus \sing(X)$ and
$(\psi^*_t)_{t\in \RR}$ on $\cN$ and to define diffeomorphisms
${\mathfrak{h}}_{y',x'}\colon \cN_{y'}\to \cN_{x'}$ for close points $x',y'\in U$.
We then conjugate these maps $\mathfrak{h}_{y',x'}$ by $I\colon \cN\to \widehat \cN$
in order to define $\widehat {\mathfrak{h}}_{y,x}:=I\circ \mathfrak{h}_{p(y),p(x)}\circ I^{-1}$
on the fibers of $\widehat \cN$.

The construction of the maps $\mathfrak{h}_{y,x}$ for close points $x,y\in U$
is similar to the definition of the non-linear Poincar\'e flow. We use the exponential maps
in order to project balls $B(0,\beta)$ in each fiber of the normal bundle $\cN$.
For identifications we use the exponential maps at $x,y$ instead of $x,\varphi_t(x)$.

As discussed in Sections~\ref{ss.NLPF} and~\ref{ss.rescaled},
since $U$ is at positive distance from $\Sing(X)$,
if $\overline \beta>0$ is small enough, for any point $x$ in a neighborhood of $U$,
the image of $B(0,\overline \beta)\subset \cN_{x}$ by the exponential map $\exp_{x}$ is transverse to the vector field $X$. Consequently, for $r_0>0$, $\beta_0\in (0,\overline \beta)$ and $\varepsilon>0$ small,
if $x,y\in U$ satisfy $d(x,y)<r_0$, then for any $u\in B(0,\beta_0)\subset \cN_{y}$,
there exists a unique time $s\in (-\varepsilon,\varepsilon)$ such that
$\varphi_s(\exp_{y}(\|X(y)\|\cdot u))$ belongs to $\exp_{x}(B(0,\overline \beta))$. After rescaling, we define
$$\mathfrak{h}_{y,x}(u):=\|X(x)\|^{-1}\cdot \exp_{x}^{-1}\circ \varphi_s\circ \exp_{y}(\|X(y)\|\cdot u).$$
Since $X$ is $C^k$, the map $\mathfrak{h}_{y,x}$ is $C^k$ also.
By the uniqueness of the parameter $s$, we obtain the relation $\mathfrak{h}_{z,x}\circ \mathfrak{h}_{y,z}=\mathfrak{h}_{y,x}$. The maps $\mathfrak{h}_{y,x}$ can be extended as diffeomorphisms
$\cN_{y}\to \cN_x$ by gluing with the affine maps $D\mathfrak{h}_{y,x}(0)$ outside a small uniform neighborhood
of $0$ in $\cN_x$. In this way the family $\{\mathfrak{h}_{y,x}\}$ varies continuously for the
$C^k$ topology, with respect to the points $x,y\in U$ such that $d(x,y)<r_0$.
\smallskip

We then check the compatibility on $U$ with the flow $({\psi_{t}^*})_{t\in\RR}$.
The Local injectivity follows immediately by choosing the time $t=s$ equal to the section time $s$ in the definition of the maps $\mathfrak{h}_{y,x}$.
Indeed if $\mathfrak{h}_{y,x}(0_{y})=\|X(x)\|^{-1}\exp_{x}^{-1}\varphi_s(y)$ has a small norm,
then $\varphi_t(y)$ has to be close to $x$
and $\widehat \varphi_t(y)$ is close to $x$.

In order to prove the Local invariance, let us take $x,y\in U$, $t\in [-2,2]$, $u\in B(0, \beta_0)\subset \cN_{y}$ such that $y$ and $\widehat{\varphi}_{t}(y)$ are $r_0$-close to $x$.
If $r_0$ has been chosen small enough the ``No small period'' Property  implies that $t$ is small.
Then the definition of the rescaled nonlinear Poincar\'e flow $(\psi^*_t)_{t\in\RR}$,
and the uniqueness of $s$ in the definition of the maps $\mathfrak{h}_{x,y}$ implies
$\mathfrak{h}_{ \varphi_t(y),x}\circ  {\psi^*_{t}}(u)=\mathfrak{h}_{y,x}(u)$. This gives the Local invariance.
The last property of Definition~\ref{d.compatible} will be checked in the next section.
 
\subsection{Global invariance}
We finally check the Global invariance in Definition~\ref{d.compatible} for the
local flow $(\psi^*_t)_{t\in\RR}$ on the open set $U$.
Let us fix $\delta,\rho>0$ small: by reducing $\delta$, one can assume that Lemma~\ref{l.flow} above holds.
One then chooses $t_0>0$ small such that $d(x,\varphi_t(x))<\delta$ for any $x\in M$ and $t\in [-t_0,t_0]$.
One fixes $r>0$ and $\beta\in (0,\delta)$ small such that:
\begin{itemize}
\item[a)]
for any $y,y'\in U$, $u\in \cN_y$, $u'\in \cN_{y'}$  with $d(y,y')<r$, $\|u\|,\|u'\|<\beta$ and
$\mathfrak{h}_y(u')=u$, then
$\varphi_t\exp_y({\|X(y)\|}u)=\exp_{y'}({\|X(y')\|}u')$ for some $t\in [-t_0,t_0]$
(as for Local injectivity);
\item[b)] $\delta,t_0, 4\beta$ satisfy the Lemma~\ref{l.project};
\item[c)] every $x\in M$ and $w\in T_xM$ with $\|w\|\le \beta\cdot\|X(x)\|$ satisfy
 $d(x,\exp_x(w))<\delta$.
\end{itemize}

Consider any $y,y'\in U$ , $u\in \cN_y$, $u'\in \cN_{y'}$ and $I,I'$ as in the statement of the Global invariance
for $(\psi^*_t)_{t\in\RR}$ on $U$.
Since $\|u\|<\beta<\delta$,
Lemma~\ref{l.flow} can be applied
to $y$ and $z:=\exp_y(\|X(y)\|\cdot u)$: for each $t\in (0,1)$, one defines $\theta_0(t)>0$ to be equal to the $t'$ given by Lemma~\ref{l.flow}.
The map $\theta_0$ is $(1+\rho)$-bi-Lipschitz and increasing. Moreover $\theta_0(0)=0$.
Since $\|\psi^*_t(u)\|< \beta<\delta$ for any $t\in I$,
one has $\psi_t(\|X(y)\|\cdot u)\in B(0,\delta\|X(\varphi_t(x))\|)$ and
one can apply inductively Lemma~\ref{l.flow} to the points $\varphi_t(y)$, $\psi_t(\|X(y)\|\cdot u)$.
This defines $\theta_0$ on $I$. One gets:
$$\forall t\in I,~~~\exp_{\varphi_t(y)}\circ \psi_t(\|X(y)\|\cdot u)=\varphi_{\theta_0(t)}(z).$$
The same argument for $y'$ and $z'=\exp_{y'}(\|X(y')\|\cdot u')$ defines a map $\theta_0'\colon I'\to \RR$.
\medskip

Let us now consider $s\in I\cap \theta_0^{-1}\circ \theta_0'(I')$.
{By (a),} since $\mathfrak{h}_y(u')=u$, there exists $t\in [-t_0,t_0]$ such that $\varphi_t(z)=z'$.
By Property (c) above, the points $\varphi_s(y)$ and $\varphi_{\theta_0(s)}(z)$ are $\delta$-close.
By our choice of $t_0$ and since $|t|\leq t_0$, the points $\varphi_{\theta_0(s)}(z)$  and $\varphi_{\theta_0(s)+t}(z)= \varphi_{\theta_0(s)}(z')$ are $\delta$-close.
Since $\theta_0(s)\in \theta'_0(I')$, by Property (c) the points $\varphi_{\theta_0(s)}(z')$ and $\varphi_{(\theta_0')^{-1}\circ\theta_0(s)}(y')$ are $\delta$-close.
Consequently, the points $\varphi_s(y)$ and $\varphi_{\theta(s)}(y')$ are $3\delta$-close, where $\theta=(\theta_0')^{-1}\circ\theta_0$.
Note that $\theta$ is bi-Lipschitz for the constant $(1+\rho)^2<1+3\rho$ and satisfies $\theta(0)=0$. This proves Part (1) of the Global invariance for the numbers $\overline\delta=3\delta$ and $\overline \rho=3\rho$.
\medskip

We now take $v\in \cN_y$, $v'=\mathfrak{h}_{y'}(v)$ in $\cN_{y'}$ such that
$\|\psi^*_{s}(v)\|<\beta$ for each $s\in I\cap \theta^{-1}(I')$.
Set $\zeta=\exp_y(\|X(y)\|\cdot v)$.
{By Lemma~\ref{l.project},} there is a unique $t'\in (-t_0,t_0)$
such that {$\varphi_{t'}(\zeta)=\exp_{y'}(w')$ for some $w'\in B(0,\delta)\subset \cN_{y'}$.}
By definition of $\mathfrak{h}_{y,y'}$, it coincides with
$\exp_{y'}(\|X(y')\|\cdot v')$.

Arguing as above, there exists $\theta_1$ such that $\theta_1(0)=0$ and
$\exp_{\varphi_s(y)}\circ \psi_s(\|X(y)\|\cdot v)=\varphi_{\theta_1(s)}(\zeta)$ for each $s\in I$.
By assumption $d(\varphi_{\theta_1(s)}(\zeta), \varphi_s(y))=\|\psi_s(\|X(y)\|\cdot v)\|$
is smaller than $\beta\cdot \|X(\varphi_s(y))\|$.
Similarly,
$$d(\varphi_{\theta_0(s)}(z),\varphi_s(y))\leq 
 \beta\cdot \|X(\varphi_s(y))\|,$$
 $$d(\varphi_{\theta'_0(s')}(z'),\varphi_{s'}(y'))\leq 
 \beta\cdot \|X(\varphi_{s'}(y'))\|.$$
 Since $z'=\varphi_t(z)$, to each $s\in I\cap \theta^{-1}(I')$
 one associates $s'$ such that $\theta_0(s)=\theta'_0(s')+t$, one gets
 $$d(\varphi_{\theta_0(s)}(z),\varphi_{s'}(y'))\leq 
 \beta\cdot\|X(\varphi_{s'}(y'))\|.$$
 If $\beta$ has been chosen small enough, one deduces $\|X(\varphi_{s}(y))\|<2 \|X(\varphi_{s'}(y'))\|$.
Combining the previous inequalities, one gets
\begin{equation}\label{e.bound-projection}
d(\varphi_{\theta_1(s)}(\zeta),\varphi_{s'}(y'))\leq 
 4\beta\cdot\|X(\varphi_{s'}(y'))\|.
 \end{equation}
By Lemma~\ref{l.project} (for $\delta, t_0, 4\beta$ as in Item (b) above),
there is $\sigma(s)\in (-t_0,t_0)$
 such that $\varphi_{\theta_1(s)+\sigma(s)}(\zeta)$ belongs to the image by $\exp_{\varphi_{s'}(y')}$ of
 $B(0,\delta\cdot \|X(\varphi_{s'}(y'))\|)\subset \cN_{\varphi_{s'}(y')}$. In the case $s'=0$, since $\sigma$ is small
 and $\mathfrak{h}_{y'}(v)=v'$,
 {the definition of $\mathfrak{h}_{y'}$ gives} $\varphi_{\theta_1(s)+\sigma(s)}(\zeta)=\exp_{y'}(\|X(y')\|\cdot v')$. 
 
 Applying Lemma~\ref{l.flow} {inductively}, one gets $\varphi_{\theta_1(s)+\sigma(s)}(\zeta)=\exp_{\varphi_{s'}(y')}(\psi_{s'}(\|X(y')\|\cdot v'))$ for any $s\in I\cap \theta^{-1}(I')$.
By~\eqref{e.bound-projection} and Lemma~\ref{l.project}, one deduces $\|\psi_{s'}(\|X(y')\|\cdot v')\|\leq \delta\cdot \|X(\varphi_{s'}(y'))\|$,
that is $\|\psi^*_{s'}(v')\|\leq \delta$, giving Part (2.i) of the Global invariance.
\medskip

We have obtained
$$\varphi_{\sigma(s)}\circ \exp_{\varphi_s(y)}\circ \psi_s(\|X(y)\|\cdot v)=\exp_{\varphi_{s'}(y')}(\psi_{s'}(\|X(y')\|\cdot v')).$$
When $\varphi_s(y)\in U$,
one deduces by definition of identification,
$$\mathfrak{h}_{\varphi_{s}(y)}\circ \psi_{s'}(\|X(y')\|\cdot v')= \psi_s(\|X(y)\|\cdot v).$$
By definition of $\theta$ and $s'$, one notices that $\theta(s)$ and $s'$ are close.
Hence $$\mathfrak{h}_{\varphi_{s'}(y')}\circ \psi_{\theta(s)}(\|X(y')\|\cdot v')=\psi_{s'}(\|X(y')\|\cdot v').$$
This gives $\mathfrak{h}_{\varphi_s(y)}\circ \psi^*_{\theta(s)}(v')=\psi^*_s(v)$, i.e. Part (2.ii). The Global invariance is now proved. \qed

\section{Negative Lyapunov exponents and periodic orbits}\label{Sec:negative-periodic}
We give in this section an application of the Main Theorem.

\begin{Proposition}
Let $X$ be a $C^1$ vector field on a closed manifold $M^d$ such that $DX(\sigma)$ is invertible at each singularity $\sigma$ and let $\mu$ be an ergodic non-atomic probability measure with $d-1$ negative Lyapunov exponents. Then $\mu$ is supported on a periodic orbit.
\end{Proposition}
\begin{proof}
Let us consider the splitting provided by Oseledets theorem
and $\mu$-almost every point $T_xM=E^s_x \oplus (\RR\cdot X(x))$.
The tangent flow restricted to $E^s$ is conjugated to the linear Poincar\'e flow.
Let us apply the Main Theorem. The measure $\mu$ is associated to an invariant measure
$\widehat \mu$ on the blownup manifold $\widehat M$.
The extended rescaled linear Poincar\'e flow $\widehat {\Psi^*}$
has only negative Lyapunov exponents, hence the same holds for the non-linear flow
$\widehat {\psi^*}$ along the $0$-section above the orbit of $\widehat\mu$-almost every point $x$.

By definition of the Lyapunov exponents, there are $s,\lambda>0$ such that
$\int \log \|\widehat {\Psi^*_s}\| d\widehat \mu =-\lambda$,
hence $\lim_{+\infty} \tfrac 1 n \sum_{i=0}^{n-1} \log \|\widehat {\Psi^*_s}(\widehat \varphi_{is}(x))\| = -\lambda$.
By Pliss lemma, we get $C>0$ such that
$$\prod_{i=0}^{n-1} \|\widehat {\Psi^*_s}(\widehat \varphi_{is}(x))\| < C \exp(-\lambda N/2), \quad \forall n\geq 0.$$
Since $\widehat \Psi^*$ is the tangent flow to $\widehat {\psi^*}$ along the $0$-section,
there exists a uniform neighborhood $B(0_x,r)$ of $0_x$ in the fiber $\widehat \cN_x$
such that $\widehat \psi^*_t$ is defined on the domain $B(0_x,r)$ and is a $1/2$-contraction for all $t>0$ large enough.

The orbit of $x$ is positively recurrent, hence there exists a large iterate
$\widehat \varphi_T(x)$ which is arbitrarily close to $x$.
Let $\mathfrak{h}$ be the identification provided by the Main Theorem.
Since $x,\widehat \varphi_T(x)$ are close, $\mathfrak{h}_{x,\widehat \varphi_T(x)}$
is $C^1$ close to the identity, hence
$\overline\psi_T:=\mathfrak{h}_{x,\widehat \varphi_T(x)}\circ \widehat \psi^*_T$ is a $1/2$-contraction of
$B(0_x,r)$ into itself and admits a fixed point $p$.
Corollary~\ref{c.closing0} implies that the forward orbit of $x$ converges to a periodic orbit of $\widehat \varphi$.
Since the orbit of $x$ equidistributes toward the measure $\widehat \mu$,
one concludes that $\widehat \mu$ is supported on a periodic orbit, and the same holds for $\mu$
and the initial flow $\varphi$.
\end{proof}


\vskip 5pt

\begin{tabular}{l l l}
\emph{\normalsize Sylvain Crovisier}
& \quad &
\emph{\normalsize Dawei Yang}
\medskip\\

\small Laboratoire de Math\'ematiques d'Orsay
&& \small School of Mathematical Sciences\\
\small CNRS - Universit\'e Paris-Saclay
&& \small Soochow University\\
\small Orsay 91405, France
&& \small Suzhou, 215006, P.R. China\\
\small \texttt{Sylvain.Crovisier@universite-paris-saclay.fr}
&& \texttt{yangdw1981@gmail.com}\\
&& \texttt{yangdw@suda.edu.cn}\
\end{tabular}


\begin{thebibliography}{XXW}	
%

%
%
%
%
%

\bibitem[BdL]{BdL}C. Bonatti and A. da Luz, Star flows and multisingular hyperbolicity. \emph{J. Eur. Math. Soc.} {\bf 23} (2001), 2649--2705.

%
%
%
%
%
%
%
%
\bibitem[CY1]{CY1} S. Crovisier and D.Yang, On the density of singular hyperbolic three-dimensional vector fields: a conjecture of Palis. {\it C. R. Math. Acad. Sci. Paris} {\bf353} (2015), 85--88.

\bibitem[CY2]{CY2} S. Crovisier and D. Yang, Homoclinic tangencies and singular hyperbolicity.  ArXiv:1702.05994.
%
%

\bibitem[GY]{GY} S. Gan and D. Yang, Morse-Smale systems and horseshoes for three dimensional flows.
\emph{Ann. Sci. \'Ecole Norm. Sup.}  \textbf{51} (2018), 39--112. 

%
%
%
%
%

\bibitem[LGW]{lgw-extended} M. Li, S. Gan and L. Wen, Robustly transitive singular sets via
approach of extended linear Poincar\'e flow. {\it Discrete Contin.
Dyn. Syst. }{\bf 13 } (2005), 239--269.

\bibitem[Li$_1$]{Lia63} S. Liao, Certain ergodic properties of a differential system on a compact differentiable manifold. {\it Acta Sci. Natur. Univ. Pekinensis} {\bf9} (1963), 241--265 and 309--326. Translated in {\it Front. Math. China} {\bf1} (2006), 1--52. 

\bibitem[Li$_2$]{Lia89} S. Liao, On $(\eta,d)$-contractible orbits of vector
fields. {\it Systems Science and Mathematical Sciences}
{\bf2} (1989), 193--227.

\bibitem[LY]{LY}Z. Lian and L-S.  Young, Lyapunov exponents, periodic orbits, and horseshoes for semiflows on Hilbert spaces. {\it J. Amer. Math. Soc.}, {\bf25}(2012), 637-665.

\bibitem[Lo]{Lo}
E. N. Lorenz, Deterministic nonperiodic flow. {\it J. Atmosph. Sci.} {\bf 20} (1963), 130--141.

%
%
%
%
%

\bibitem[No]{No} A. Nobile, Some properties of the Nash blowing-up.
\emph{Pacific J. Math.} \textbf{60} (1975), 297--305.

%
%
%
%
%
%
%
%
%

\bibitem[T]{takens} F. Takens, Singularities of vector fields. \emph{Publ. Math. Inst. Hautes \'Etudes Sci.} \textbf{43} (1974), 47--100.

\bibitem[Y]{Y}J. Yang, Cherry flow: physical measures and perturbation theory. \emph{Ergodic Theory Dynam. Systems} \textbf{37}(2017), 2671--2688.


%
%
%
%
%
%
%
%
%
%
%
%
%
%
%
%
%
%
%
%
%
%
%
%
%
%
%
%
%
%




%
%
%
%
%
%
%
%
%
%
%
%
%
%
%
%
%
%
%
%
%
%
%
%
%
%
%
%
%
%
%
%
%
%
%
%
%
%
%
%
%
%
%
%
 
\end{thebibliography}
\end{document}